\documentclass[reqno,centertags, 12pt]{amsart}
\usepackage{amsmath,amsthm,amscd,amssymb}
\usepackage{esint}  %from Maxim for circular integral
\usepackage{latexsym}
\usepackage{graphicx}
\usepackage{enumitem}
\usepackage[mathscr]{eucal}
\usepackage{hyperref}
\newcommand{\bbC}{{\mathbb{C}}}
\newcommand{\bbD}{{\mathbb{D}}}

\newcommand{\bbR}{{\mathbb{R}}}
\newcommand{\bbT}{{\mathbb{T}}}

\newcommand{\calB}{{\mathcal{B}}}
\newcommand{\calC}{{\mathcal{C}}}

\newcommand{\calG}{{\mathcal G}}
\newcommand{\calH}{{\mathcal H}}

\newcommand{\calK}{{\mathcal K}}
\newcommand{\calL}{{\mathcal L}}

\newcommand{\calP}{{\mathcal P}}

\newcommand{\bdone}{{\boldsymbol{1}}}

\newcommand{\lb}{\label}

\newcommand{\ti}{\tilde  }

\newcommand{\tr}{\text{\rm{Tr}}}

\newcommand{\ran}{\text{\rm{ran}}}

\newcommand{\bi}{\bibitem}

\newcommand{\beq}{\begin{equation}}
\newcommand{\eeq}{\end{equation}}
\newcommand{\ba}{\begin{align}}
\newcommand{\ea}{\end{align}}

\let\det=\undefined\DeclareMathOperator{\det}{det}

\newcounter{smalllist}

\newcommand{\comm}[1]{}
\allowdisplaybreaks
\numberwithin{equation}{section}
\newtheorem{theorem}{Theorem}[section]

\newtheorem*{p2.1}{Proposition 2.1}
\newtheorem{proposition}[theorem]{Proposition}
\newtheorem{lemma}[theorem]{Lemma}

\theoremstyle{definition}
\newtheorem*{remark}{Remark}
\newtheorem*{remarks}{Remarks}

\newcommand{\jap}[1]{\langle #1 \rangle}
\newcommand{\norm}[1]{\lVert#1\rVert}
\newcommand{\Norm}[1]{\lVert#1\rVert}

\begin{document}

\title[OPUC and Poncelet's Theorem]{Poncelet's Theorem, Paraorthogonal Polynomials and the Numerical Range of Compressed Multiplication Operators}
\author[A.~Mart\'{i}nez--Finkelshtein, Brian~Simanek and B.~Simon ]{Andrei~Mart\'{i}nez--Finkelshtein$^{1,2}$, Brian~Simanek$^{3}$ \\and Barry Simon$^{4,5}$}

\thanks{$^1$ Department of Mathematics, Baylor University, Waco, TX 76798, USA, and  Department of Mathematics, University of Almer\'{\i}a, Almer\'{\i}a 04120, Spain.  E-mail: A\_Martinez-Finkelshtein@baylor.edu}

\thanks{$^2$ Research supported in part by the Spanish Government -- European Regional Development Fund (grant MTM2017-89941-P), Junta de Andaluc\'{\i}a (research group FQM-229 and Instituto Interuniversitario Carlos I de F\'{\i}sica Te\'orica y Computacional), and by the University of Almer\'{\i}a (Campus de Excelencia Internacional del Mar  CEIMAR)}

\thanks{$^3$ Department of Mathematics, Baylor University, Waco, TX 76798, USA. E-mail: Brian\_Simanek@baylor.edu}

\thanks{$^4$ Departments of Mathematics and Physics, Mathematics 253-37, California Institute of Technology, Pasadena, CA 91125, USA. E-mail: bsimon@caltech.edu}

\thanks{$^5$ Research supported in part by NSF grant DMS-1665526 and in part by Israeli BSF Grant No. 2014337.}

\

\date{\today}

\begin{abstract}
	
There has been considerable recent literature connecting Poncelet's theorem to ellipses, Blaschke products and numerical ranges, summarized, for example, in the recent book \cite{GorkinBk}.  We show how those results can be understood using ideas from the theory of orthogonal polynomials on the unit circle (OPUC) and, in turn, can provide new insights to the theory of OPUC.

\end{abstract}

\maketitle

%%%%%%%%%%%%%%%%%%%%%%%%%%%%%%%%%%%%%%%%%%%%%%%%%%%%%%%%%%%%%%
\section{Introduction} \lb{s1}
%%%%%%%%%%%%%%%%%%%%%%%%%%%%%%%%%%%%%%%%%%%%%%%%%%%%%%%%%%%%%%

In 1813, Jean--Victor Poncelet \cite{Ponc} proved a remarkable theorem (see Halbeisen--Hunderb\"{u}ler \cite{HH} for a simple proof) that says if $K$ is an ellipse inside another ellipse, $Q$, so that there is a triangle with vertices in $Q$ and sides tangent to $K$, then there are infinitely many such triangles, indeed, so many that their vertices fill $Q$ and their tangent points fill $K$.  There has been a huge literature motivated by this gem of projective geometry, even a recent book \cite{PonceBk}.  Our paper studies three different related developments.

Marden's 1948 book \cite[Ch.~1, \S 4]{Marden}, \emph{Geometry of Polynomials}, popularized a theorem which he traces back to an 1864 paper of J\"{o}rg Siebeck   \cite{Sieb}.

\renewcommand\thetheorem{\Alph{theorem}}

\begin{theorem} \lb{TA} Let $\{w_j\}_{j=1}^p$ be the vertices of a convex polygon in $\bbC$ ordered clockwise.  Let $m_j\in\bbR$, and let

\begin{equation}\label{1.1}
  M(z)=\sum_{j=1}^{p} \frac{m_j}{z-w_j}
\end{equation}

Then the zero's of $M$ are the foci of a curve of class $p-1$ which intersects each of the line segments $w_jw_k;\, j, k=1,\dots,p;\, j\ne k$ at the point dividing the line in ratio $m_j/m_k$.
\end{theorem}

In this brief introduction, we are not going to try to give you the rather complicated definitions of the foci of a curve or of class  nor will we use these notions later in this paper (but see \cite{Langer,Singer}).  We state this theorem to emphasize there is a $n$--gon version, that $M$ and its zeros play a special role and that the ratios $m_j/m_{j+1}$ occur.

The second set of results concern Blaschke products.  Starting in 2002, Daepp, Gorkin and collaborators wrote a series of papers \cite{Gorkin2002,Gorkin2017A,Gorkin2010,Gorkin2008,Gorkin2011,Gorkin2017B} considering finite Blaschke products\footnote{ When we want to specify $n$, we will refer to \eqref{1.2} as \textit{$n$ fold Blaschke products}. } of the form (for $\{z_j\}_{j=1}^n\subset\bbD:= \{ z\in \bbC:\, |z|<1\}$, maybe not be all distinct)

\begin{equation}\label{1.2}
  B(z) = \prod_{j=1}^{n}\frac{z-z_j}{1-\bar{z}_jz}
\end{equation}

These are precisely the Schur functions (analytic maps of $\bbD$ to itself) which are analytic in a neighborhood of $\overline{\bbD}$, of magnitude $1$ on $\partial\bbD$, with $n$ zeros (they actually consider $zB$ and sometimes divide their basic function by $z$; we prefer to take this $B$ and sometime multiply it by $z$). Since $|zB(z)|<1$ on $\bbD$ and $|zB(z)|= 1$ on $\partial\bbD$, by the Cauchy--Riemann equations, the map $e^{i\theta} \mapsto e^{i\theta}B(e^{i\theta})$ is strictly increasing in $\theta$ and by the argument principle is $n+1$ to $1$ so, for each $\lambda\in\partial\bbD$, there exist $n+1$ solutions, $w_j;\, j=1,\dots,n+1$, of $wB(w) = \bar{\lambda}$ (they take $\lambda$; it will be clear later why we prefer $\bar{\lambda}$ as our label).  We label the $w_j$ with increasing arguments where arguments are taken in $[0,2\pi)$.

The main result in this approach is

\begin{theorem} \lb{TB} For any $\{z_j\}_{j=1}^n\subset\bbD$ and any $\lambda\in\partial\bbD$, there exist $m_j(\lambda) >0$ with $\sum_{j=1}^{n+1} m_j(\lambda)=1$ so that

\begin{equation}\label{1.3}
    \sum_{j=1}^{n+1} \frac{m_j(\lambda)}{z-w_j} =\frac{B(z)}{zB(z)-\bar{\lambda}}
\end{equation}
\end{theorem}

The right side of this expression is a rational function of $z$ which is $z^{-1}+\mbox{O}(|z|^{-2})$ at infinity and with poles exactly at the $w_j$ so the left side is just a partial fraction expansion and $\sum_{j=1}^{n+1} m_j(\lambda)=1$ follows from the asymptotics at infinity.  The main issue is the proof of Daepp et al \cite{Gorkin2002} that $m_j > 0$ and they proved this by finding an explicit formula for the $m_j$ in terms of the $z$'s and $w$'s,
\begin{equation}\label{1.4}
  m_j = \left[1+\sum_{k=1}^{n+1} \frac{1-|z_k|^2}{|w_j-z_k|^2}\right]^{-1}
\end{equation}
It is left unmentioned that there is a probability measure and what its significance is.  It has been noted (see, for example, footnote 5 on page 107 in \cite{Gorkin2019}) that there is a converse of sorts to this result, that is, for every $\{m_j\}_{j=1}^{n+1}$ with $m_j  >0$ and $\sum_{j=1}^{n+1} m_j =1$, there is a Blaschke product so that \eqref{1.3} holds.

The following theorem is natural to state in this $B(z)$ language

\begin{theorem} \lb{TC} Fix $\lambda\ne\mu$ both in $\partial\bbD$ and let $\{w_j\}_{j=1}^{n+1}$ (resp. $\{u_j\}_{j=1}^{n+1}$) be the solutions of $zB(z) = \bar{\lambda}$ (resp. $zB(z) = \bar{\mu})$.  Then the $w's$ and $u's$ interlace. Conversely, if one is given such interlacing sets, there is a unique $n$ fold Blaschke product so that the $w$'s and $u$'s are the solutions of a $zB(z)$ equation.
\end{theorem}

This result was first proven by Gao and Wu \cite{GW2004A} in the $S_n$ framework below and the $w$'s and $u$'s enter as vertices of Poncelet $(n+1)$--gons. Their proof is long and involves lots of manipulations of determinants. The later, much shorter proof, of Daepp, Gorkin and Voss \cite{Gorkin2010} constructs some rational Herglotz functions with given interlacing zeros and poles.  We have a simple third proof.  For reasons that will become obvious later, for now, we'll call this Wendroff's Theorem for Blaschke products. Parameter counting for this theorem is a little subtle.  The set of $w$'s and $u$'s lie in a $2n+2$ real dimensional manifold while it appears the equivalent set is only the $n$ $z_j$'s in $\bbD$ which is only $2n$ real parameters.  But to get the $w$'s and $u$'s one needs two additional free parameters, namely $\lambda$ and $\mu$.  Conversely these parameters are determined by the $w$'s and $u$'s since
\begin{equation}\label{1.5}
   \lambda=-\prod_{j=1}^{n+1} (-\bar{w}_j) \qquad \mu=-\prod_{j=1}^{n+1} (-\bar{u}_j)
\end{equation}

The final theme concerns a class of finite dimensional matrices, now called $S_n$, studied in a series of independent papers by Gau--Wu \cite{GW1998,GW1999,GW2003,GW2004A,GW2004B,GW2004C,GW2013}, also \cite{Wu00}, and by Mirman \cite{Mirman1998,Mirman2003,Mirman2009,Mirman2001,Mirman2005} both series starting in 1998.  Recall that an operator on a Hilbert space is called a \emph{contraction} if its norm is at most $1$. It is called \emph{completely non--unitary} if it has no invariant subspace on which it is unitary, which in the finite dimensional case, is equivalent to there being no eigenvector with eigenvalue $\lambda$ obeying $|\lambda|=1$.  In the finite dimensional case, the \emph{defect index} of a contraction, $A$, is defined to be the dimension of the range of $\bdone - A^*A$.  The space $S_n$ is the set of completely non--unitary contractions on $\bbC^n$ with defect index $1$.  One important theorem is

\begin{theorem} \lb{TD} For any $\{z_j\}_{j=1}^n\subset\bbD$ (maybe not different), there is an operator $A\in S_n$ whose eigenvalues (counting algebraic multiplicity) are the $z_j$.   Any two elements in $S_n$ are unitarily equivalent if and only if they have the same eigenvalues and multiplicities.
\end{theorem}

Recall that if $A$ is an operator on a Hilbert space, $\calH$, then $N(A)$, the \emph{numerical range} of an operator, $A$, on $\calH$ is the set of values $\jap{\varphi,A\varphi}$ where we run through all $\varphi\in\calH$ with $\Norm{\varphi}=1$ (not $\le 1$!). It is a subtle fact that $N(A)$ is a convex subset of $\bbC$ and an easy fact that it is compact when $\calH$ is finite dimensional.  See \cite{BD,GR,HJ} for more on numerical ranges.  It is not hard to show that if $A$ is normal, then (the closure of) $N(A)$ is the closed convex hull of the spectrum, so, in the finite dimensional normal case, $N(A)$ is the convex hull of the eigenvalues and so a convex polygon.  In particular, if $A$ is a $k$ dimensional unitary operator with simple spectrum, then $N(A)$ is a convex $k$--gon inscribed in $\partial\bbD$.

Let $\calH \subset \calK$ be two Hilbert spaces and $P$ the orthogonal projection from $\calK$ onto $\calH$. If $A \in\calL(\calH),\, B\in\calL(\calK)$, we are interested in the relation $A=PBP\restriction \calH$. If that holds we say that $A$ is a \emph{compression} of $B$ and that $B$ is a \emph{dilation} of $A$.  In case $\dim(\calK) < \infty$, we call $\dim(\calK)-\dim(\calH)$ the \emph{rank} of the dilation. There is a huge literature on dilations, much of it involving families rather than single operators, for example, see \cite{Nagy1,Nagy2}. Our usage limiting to a single operator and not demanding $A^n=PB^nP\restriction\calH$ (for $n=1,\dots$) is more common; see, for example, the Wikipedia article $https://en.wikipedia.org/wiki/Dilation\_(operator\_theory)$.

Given a contraction, $A$ on $\calH$, one is interested in finding $\calK$ and $B\in\calL(\calK)$ so that $B$ is a unitary dilation of $A$.  It is easy to construct such a dilation on $\calK=\calH\oplus\calH$, so if $\dim(\calH)=n$, a rank $n$ unitary dilation, but one can show there is a one parameter family of rank one unitary dilations, $\{B_\lambda\}_{\lambda\in\partial\bbD}$ of any $A\in S_n$. For different $\lambda$, the eigenvalues are different and every point on the circle is an eigenvalue of exactly one $B_\lambda$.

The big theorem in these papers of Gau--Wu and Mirman is

\begin{theorem} \lb{TE} Let $A \in S_n$ and $\{B_\lambda\}_{\lambda\in\partial\bbD}$ its rank one unitary dilations.  For each fixed $\lambda$, $N(B_\lambda)$ is a solid $n+1$--gon with vertices $\{ w_j \}_{j=1}^{n+1}$ on $\partial\bbD$.  Each edge of this polygon is tangent to $N(A)$ at a single point and as $\lambda$ moves through all of $\partial\bbD$, these tangent points trace out the entire boundary of $N(A)$.  Moreover
\begin{equation}\label{1.6}
  N(A) = \bigcap_{\lambda\in\partial\bbD} N(B_\lambda)
\end{equation}
If, for a fixed $\lambda$, one forms
\begin{equation}\label{1.7}
   M(z)=\sum_{j=1}^{n+1} \frac{m_j}{z-w_j}
\end{equation}
for suitable $m_j$, then the zeros of $M$ are precisely the eigenvalues of $A$.
\end{theorem}

\begin{remark} It was a conjecture of Halmos \cite{Halmos}, proven by Choi--Li \cite{Choi}, that a formula like \eqref{1.6} holds for any contraction if the intersection is over all dilations of rank at most $n$.
\end{remark}

We are not the first people to realize the relations between these themes; indeed, there is a very recent book on the subject \cite{GorkinBk} entitled:   \emph{Finding Ellipses: What Blaschke Products, Poncelet's Theorem and the Numerical Range Know about Each Other}.  The point of our paper is that while the authors of these works didn't know it, they were discussing aspects of the theory of orthogonal polynomials on the unit circle (OPUC).  Our realization of this additional connection will allow us to give new proofs and sometimes extend these results.  These proofs are often quite simple (if one knows the OPUC background \cite{OPUC1, OPUC2, OPUCOneFt}!) and sometimes quite illuminating.  For example, we'll see that the left side of \eqref{1.3} is the  matrix element of the resolvent of a unitary operator and the right side a Cramer's rule ratio of determinants!  Moreover, as we'll explain, this formula, which is one of the main results of the 2002 paper of Daepp, Gorkin and Mortini \cite{Gorkin2002}, can be viewed as a special case of a general OPUC result of Khrushchev \cite{Khr2001} published the year before!

We will also find that these lovely earlier ideas provide new results in the theory of OPUC of interest within that literature.

Because this paper is aimed at two disparate audiences (namely workers on the themes above and workers in OPUC) with rather different backgrounds, the presentation is more discursive than it might be if directed only at experts in a single area.  In particular, Section \ref{s2} which officially sets notation and terminology presents a lot of results in OPUC with references to proofs.  Section \ref{s3} states our main results.  Some of these results will rely on a new theorem presented in Section \ref{s4} on GGT matrices.  Sections \ref{s5}--\ref{s7} recover results on $S_n$ and Sections \ref{s8}--\ref{s9} on Blaschke products.  Section \ref{s10} discusses two new variants of Wendroff's Theorem connected to Theorem \ref{TC} and the final three sections have additional remarks and results.

B.S. would like to thank Fritz Gesztesy and Lance Littlejohn for the invitation to visit Baylor where our collaboration was begun.

%%%%%%%%%%%%%%%%%%%%%%%%%%%%%%%%%%%%%%%%%%%%%%%%%%%%%%%%%%%%%%
\section{OPUC On One Toe} \lb{s2}
%%%%%%%%%%%%%%%%%%%%%%%%%%%%%%%%%%%%%%%%%%%%%%%%%%%%%%%%%%%%%%

\renewcommand\thetheorem{\arabic{theorem}}
\numberwithin{theorem}{section}

A probability measure, $d\mu$ on $\partial\bbD$ is called \emph{non-trivial} if it is not supported on a finite set of points.  That is true if and only if $L^2(\partial\bbD, d\mu)$ is infinite dimensional and, in turn, is true if and only if $\{z^j\}_{j=0}^\infty$ are linearly independent as functions in that $L^2$ space.  In that case, by using Gram--Schmidt on this set, we can define monic orthonormal polynomials, $\Phi_n(z;d\mu)$, and orthonormal polynomials, $\varphi_n(z;d\mu)=\Phi_n(z;d\mu)/\Norm{\Phi_n}$.  Usually, we'll drop the ``$d\mu$'' unless it is needed for clarity.

The now standard reference for OPUC is Simon \cite{OPUC1, OPUC2}; older references are parts of the books of Szeg\H{o} \cite{SzegoBk}, Geronimus \cite{GerBk} and Freud \cite{FrBk}.  A summary of the high points is \cite{OPUCOneFt} which was named after a story from the Talmud.  This even briefer summary is to set notation and terminology and emphasize those aspects of the theory we'll need.  We note that large swathes of the standard theory concern asymptotics of the $\Phi$ 's and the relation of the Verblunsky coefficients to qualitative properties of the measure, none of which are relevant to our study here which concerns only $\{\Phi_n\}_{n=0}^N$ for fixed finite $N$.  We note that there are issues we hope to study elsewhere of the large $N$ behavior of the numerical range where these ideas may be important.

On $L^2(\partial\bbD, d\mu)$, one can define the antiunitary maps
\begin{equation}\label{2.1}
  \tau_n(f)(e^{i\theta}) = e^{in\theta} \overline{f(e^{i\theta})}
\end{equation}
which takes $z^k$ to $z^{n-k}$.  Let $\calP_n$ be the $n+1$ dimensional space of complex polynomials of degree at most $n$.  Then $\tau_n$ maps $\calP_n$ to itself and has the form:
\begin{equation}\label{2.2}
  \tau_n(P_n(z)) = z^n \overline{P_n\left(\frac{1}{\bar{z}}\right)}
\end{equation}
so it just reverses the coefficients of $P_n$ and complex conjugates them.  We will follow the awful, but unfortunately universal, convention of usually writing $P_n^*$ instead of $\tau_n(P_n)$ hoping the implicit $n$ is clear so for example one writes
\begin{equation*}
  \left(z\Phi_n\right)^* = \Phi_n^*
\end{equation*}
where the $^*$ on the left is $\tau_{n+1}$ while the one on the right is $\tau_n$!

This map is important because of the second part of the elementary

\begin{proposition} \lb{P2.1} In $\calP_n$

(a) Any $f$ orthogonal to $\{z^j\}_{j=0}^{n-1}$ is a multiple of $\Phi_n$

(b) Any $f$ orthogonal to $\{z^j\}_{j=1}^{n}$ is a multiple of $\Phi_n^*$
\end{proposition}

Because we will need ideas in its proof later we'll give the simple proof of the basic \emph{Szeg\H{o} recursion}

\begin{theorem} \lb{T2.2} Let $d\mu$ be a non--trivial probability measure on $\partial\bbD$.  Then for each $n=0,1,\dots$ there exists $\alpha_n(d\mu)\in\bbD$ so that
\begin{equation}\label{2.3}
  \Phi_{n+1}(z) = z\Phi_{n}(z)-\overline{\alpha}_n\Phi_n^*(z)
\end{equation}
\end{theorem}

\begin{proof} Because they are monic, $\Phi_{n+1}(z) - z\Phi_{n}(z)$ is a polynomial of degree $n$ which it is easy to see is orthogonal to $\{z^j\}_{j=1}^{n}$ and so a multiple of $\Phi_n^*$ which proves that \eqref{2.3} holds for some $\alpha_n\in\bbC$.  By moving the second term on the right to the left and noting that the two terms now on the left are orthogonal one sees that $\Norm{\Phi_{n+1}}^2+ \Norm{z\Phi_{n}}^2 =  |\alpha_n|^2\Norm{\Phi_n^*}^2$ so one has that
\begin{equation}\label{2.4}
  \Norm{\Phi_{n+1}} = \rho_n \Norm{\Phi_n}; \qquad \rho_n \equiv \sqrt{1-|\alpha_n|^2}
\end{equation}
proving that $|\alpha_n|<1$.
\end{proof}

The $\alpha_n$ are called \emph{Verblunsky coefficients}.  The strange sign and complex conjugate in \eqref{2.3} are picked so that Theorem \ref{T2.5} below is true.   Taking $z=0$ in \eqref{2.3} and noting that $\Phi_n$ monic implies that $\Phi_n^*(0)=1$, we see that
\begin{equation}\label{2.5}
  \alpha_n = -\overline{\Phi_{n+1}(0)}
\end{equation}

\eqref{2.4} implies the important
\begin{equation}\label{2.6}
  \Norm{\Phi_n} = \rho_0\dots\rho_{n-1}
\end{equation}

Applying $\tau_{n+1}$ to \eqref{2.3} we get $\Phi_{n+1}^*(z) = \Phi_{n}^*(z)-\alpha_n\Phi_n(z)$.  This equation and \eqref{2.3} can be inverted to give what are called inverse Szeg\H{o} recursion, which we'll write for the normalized OPUC:
\begin{align}
  z\varphi_n(z) &= \rho_n^{-1}(\varphi_{n+1}(z) + \bar{\alpha}_n\varphi_{n+1}^*(z)) \lb{2.7} \\
  \varphi_n^*(z) &= \rho_n^{-1}(\varphi_{n+1}^*(z) + \alpha_n\varphi_{n+1}(z)) \lb{2.8}
\end{align}
An important basic fact about zeros of OPUC is the following (which the reader is right to suspect is connected to Theorem \ref{TD}!) is

\begin{theorem} [Wendroff's Theorem for OPUC] \lb{T2.3}  All the zeros of $\Phi_n(z)$ lie in $\bbD$.  Conversely, given any labelled set of $n$, not necessarily distinct, points in $\bbD$, there exists a measure  so that those points are exactly the zeros (counting multiplicity) of the associated $\Phi_n(z)$.  Any two such measures have the same $\{\alpha_j\}_{j=0}^{n-1}$ and so also the same $\{\varphi_j\}_{j=0}^n$.
\end{theorem}

\begin{remarks} 1.  The first statement goes back to Szeg\H{o}'s basic 1920-21 OPUC paper \cite{SzegoBasic}. \cite{OPUC1} has at least six proofs of this statement: \cite[proofs of Theorem 1.7.1 containing equations (1.7.22), (1.7.43), (1.7.45), (1.7.46), (1.7.47) and (1.7.51)]{OPUC1}. More generally, Fej\'{e}r \cite{Fejer} proved for any measure of compact support in $\bbC$ whose support doesn't lie in a line that the zeros of all its OPs lie in the interior of the convex hull of the support.

2.  The full result goes back to Geronimus \cite{Geron1946} in 1946 long before the proofs of Theorem \ref{TD}.

3.  The name comes from a theorem for orthogonal polynomials on the real line (OPRL) proven in Wendroff \cite{Wendroff} in 1961: given 2n+1 distinct points in $\bbR$ thought of as an $n$ point set interlacing an $n+1$ point set, there is a probability measure on $\bbR$ with all moments finite so that the two sets are the zeros of the OPRL $P_n$ and $P_{n+1}$ and all such measures have the same first $2n+1$ Jacobi parameters (i.e., coefficients of the three terms recurrence relation) and $\{p_j\}_{j=0}^{n+1}$, \cite[Subsection 1.2.6]{OPUC1}.  The OPRL result without proof appears in a footnote of Geronimus' paper and \cite{OPUC1} tried to push the name Geroniums--Wendroff Theorem for both results but the literature seems to have stuck with Wendroff which we'll follow in this paper for theorems of this type including new ones.  Theorem \ref{TC} has the flavor of Wendroff's OPRL result and we'll see it can be view as an analog for paraorthogonal polynomials.

4.  The uniqueness part, namely that $\Phi_n$ determines the $\{\alpha_j\}_{j=0}^{n-1}$ comes from \eqref{2.5} which determines $\alpha_{n-1}$, and, then from inverse recursion for the $\Phi$'s, we get $\Phi_{n-1}$ and so by iteration all the $\{\alpha_j\}_{j=0}^{n-1}$.

5.  One way to show existence is a calculation \cite[Proof of Theorem 1.7.5]{OPUC1} that proves that if $Q_n$ has all its zeros in $\bbD$, and if $d\mu(\theta) = c\, d\theta /|Q(e^{i\theta})|^2$ (where $c$ is a normalization constant) then $Q_n$ is orthogonal in $L^2(d\mu)$ to $\calP_{n-1}$ and so is the monic OPUC as required.  This is the measure with $Q_n$ as monic OPUC that has $\alpha_j=0$ if $j \ge n$.
\end{remarks}

Two classes of functions (discussed in detail in \cite[Section 1.3]{OPUC1}) are \emph{Carath\'{e}odory functions} (analytic function on $\bbD$ which obey $\mbox{Re}(F(z)) >0; \, F(0)=1$) and \emph{Schur functions} (analytic functions on $\bbD$ which obey $|f(z)| < 1$).  Given a probability measure, $d\mu$, on $\partial\bbD$, we define two associated functions on $\bbD$:
\begin{equation}\label{2.9}
  F(z)=\int \frac{e^{i\theta}+z}{e^{i\theta}-z}\,d\mu(\theta); \qquad F(z) = \frac{1+zf(z)}{1-zf(z)}
\end{equation}
called the Carath\'{e}odory function and Schur function of $d\mu$.

Schur associated a set of parameters to any Schur function via $f_0\equiv f$ and
\begin{equation}\label{2.10}
  \gamma_n(f) = f_n(0); \qquad f_n(z) = \frac{\gamma_n+zf_{n+1}(z)}{1-\bar{\gamma}_nf_{n+1}(z)}
\end{equation}
If $f$ is a finite degree $m$ Blaschke product, then $\gamma_m\in\partial\bbD$ and the process terminates.  If not (in which case we call $f$ a \emph{non--trivial Schur function}), we can define the \emph{Schur iterates}, $f_n$, and \emph{Schur parameters}, $\gamma_n(f)\in\bbD$, for all $n$.

\begin{theorem} [Schur's Theorem (1917)] \lb{T2.4} There is a one--one correspondence between non--trivial Schur functions and sequences $\{\gamma_n\}_{n=1}^\infty$ in $\bbD$ given by the map from $f$ to its Schur parameters.
\end{theorem}

This was proven by Schur \cite{Schur}; see \cite[Section 1.3.6]{OPUC1}

\begin{theorem} [Verblunsky's Theorem (1935)] \lb{T2.5} There is a one--one correspondence between non--trivial probability measures on $\partial\bbD$ and sequences $\{\alpha_n\}_{n=1}^\infty$ in $\bbD$ given by the map from a measure to its OPUC and the Verblunsky coefficients defined via Szeg\H{o} recursion.
\end{theorem}

This was proven by Verblunsky \cite{Verb} using an equivalent definition of his coefficients.  \cite{OPUC1} has four proofs of this result (see \cite[Theorems 1.7.11, 3.1.3, 4.1.5 and 4.2.8]{OPUC1}.

\begin{theorem} [Geronimus' Theorem (1944)] \lb{T2.5} Let $d\mu$ be a non--trivial probability measure on $\partial\bbD$ and $f$ its Schur function.  Then
\begin{equation}\label{2.11}
  \alpha_n(d\mu) = \gamma_n(f)
\end{equation}
\end{theorem}

This theorem is due to Geronimus \cite{Geron1944}.  \cite{OPUC1} has five proofs of this result (see \cite[Theroems 3.1.4, 3.2.7, 3.2.10, 3.4.3 and 4.5.9]{OPUC1}.  This theorem explains why one writes Szeg\H{o} recursion with the complex conjugate and minus sign on $\alpha$.

Let $d\mu$ be a non--trivial probability measure with Verblunsky coefficients $\{\alpha_n\}_{n=1}^\infty$. The \emph{second kind polynomials} for $d\mu$, written $\{\Psi_n\}_{n=0}^\infty$, are the OPUC for the measure whose Verblunsky coefficients are $\{-\alpha_n\}_{n=1}^\infty$.  The following was given explicitly in Geronimus \cite{Geron1944, Geron1946} and earlier implicitly in Verblunsky \cite{Verb}.

\begin{theorem} \lb{T2.6} Let $d\mu_n$ be the measure with Verblunsky coefficients
\begin{equation}\label{2.12}
  \alpha_j(d\mu_n) = \left\{
                       \begin{array}{ll}
                         \alpha_j(d\mu), & \hbox{ if } j \le n-1\\
                         0, & \hbox{ if } j \ge n
                       \end{array}
                     \right.
\end{equation}
Then the Carath\'{e}odory function, $F_n$ of $d\mu_n$ is $\Psi_n^*/\Phi_n^*$, and the Carath\'{e}odory function, F, of $d\mu$ is $\lim_{n\to\infty} F_n$ uniformly on compact subsets of $\bbD$.
\end{theorem}

The proof (see \cite[(3.2.21)]{OPUC1}) depends on the formula
\begin{equation}\label{2.13}
  \Psi_k^*(z)\Phi_k(z) + \Phi_k^*(z)\Psi_k(z) = 2z^k \prod_{j=0}^{k-1} \rho_j^2
\end{equation}

In 2001, Khrushchev \cite{Khr2001} found a new approach to Rahkmanoff's Theorem (which gives a sufficient condition for the Verblunsky coefficients to be asymptotically vanishing) with lots of other interesting stuff.  A basic result he needed was the following

\begin{theorem} [Khrushchev's formula] \lb{T2.7} Let $f$ be the Schur function of some non--trivial probability measure, $d\mu$, on the unit circle and let $f_n$ be its $n^{th}$ Schur iterate. Let $B_n(z) = \Phi_n(z)/\Phi_n^*(z)$. Then the Schur function of the probability measure $|\varphi_n(e^{i\theta})|^2 d\mu$ is $f_n(z)B_n(z)$.
\end{theorem}

\cite{OPUC1, OPUC2} have three proofs of this: Khrushchev's original proof \cite[Theorem 9.2.2]{OPUC2}, a proof using second kind polynomials \cite[Corollary 4.4.2]{OPUC1} and a proof using rank two perturbations of CMV matrices \cite[Theorem 4.5.10]{OPUC1}.

This formula is an OPUC analog of the fact that the Green's function for a whole line Schr\"{o}dinger operator is the product of two suitably normalized Weyl solutions.

We let $P_n$ be the projection in $L^2(d\mu)$ onto $\calP_{n-1}$, the polynomials of degree at most $n-1$.  We use the subscript $n$ because operators on  $\calP_{n-1}$ are represented by $n \times n$ matrices.  Given a non-trivial measure, let $M_z$ be multiplication on by $z$ on $L^2(\partial\bbD,d\mu)$.  By a \emph{compressed multiplication operator}, we mean the compression of the unitary $M_z$ to polynomials of degree at most $n-1$, i.e. $A=P_nM_zP_n$ restricted to $\calP_{n-1}$.  We'll write $A^\mu$ when we want to be explicit about the measure.

\begin{theorem} \lb{T2.8} $A^\mu$ depends only on the Verblunsky coefficients $\{\alpha_j(d\mu)\}_{j=0}^{n-1}$.  That is
\begin{equation}\label{2.14}
  A^\mu=A^\nu \iff \forall_{0 \le j \le n-1} \alpha_j(d\mu) = \alpha_j(d\nu)
\end{equation}
\end{theorem}

\begin{remarks} 1. We don't merely mean unitarily equivalent.  These are operators on an explicit space of polynomials and we mean equality.  That said, we'll see shortly that if we only know that $A^\mu$ is unitarily equivalent to $A^\nu$, it is still true that $\forall_{0 \le j \le n-1} \alpha_j(d\mu) = \alpha_j(d\nu)$.

2.  While this result is not unexpected, we sketch the proof since we don't know any precise references.
\end{remarks}

\begin{proof}  If the relevant Verblunsky coefficients are equal, then they define the same set of $\{\varphi_j\}_{j=0}^{n-1}$ and $\{\varphi_j^*\}_{j=0}^{n-1}$ and the first is an orthonormal basis.  Moreover, they both have
\begin{equation}\label{2.15}
  A\varphi_j = \left\{
                 \begin{array}{ll}
                   \rho_j\varphi_{j+1} + \bar{\alpha}_j\varphi_j^* & \hbox{ if } 0 \le j \le n-2 \\
                   \bar{\alpha}_j\varphi_j^* & \hbox{ if } j = n-1
                 \end{array}
               \right.
\end{equation}
and so they are equal.

Conversely, for any $\mu$, we have that $Q_{\mu}((A^\mu)^{n-1}\bdone) = \Phi_{n-1}(\cdot;\mu)$ where $Q_\mu$ is the projection onto the orthogonal complement of $\{(A^\mu)^j\bdone\}_{j=1}^{n-2}$ so, by Wendroff's Theorem for OPUC, equality of the $A$'s  implies equality of the $\alpha_j$'s for $0 \le j \le n-2$.  Then \eqref{2.15} for $j=n-1$ implies the equality of the $\alpha_{n-1}$'s.
\end{proof}

Fix $z_0$. When does $(A^\mu-z_0)q=0$ have a non--zero solution $q \in \calP_{n-1}$?  Clearly, only if $(z-z_0)q(z)$, which is a degree $n$ polynomial, is killed by $P_n$, i.e. if $(z-z_0)q(z)$ is a multiple of $\Phi_n(z)$ and this happens if and only if $\Phi_n(z_0)=0$.  Thus the eigenvalues are precisely the zeros of $\Phi_n$.  A closer look shows that all the eigenvalues have geometric multiplicity $1$ and algebraic multiplicity the order of the zero.  Thus:

\begin{theorem} \lb{T2.9}  The eigenvalues of $A$ are the zeros of $\Phi_n$ including up to algebraic multiplicity.  Thus
\begin{equation}\label{2.16}
  \det(z-A) = \Phi_n(z)
\end{equation}
In particular, since for an $n\times n$ matrix, $C$, $\det(-C)=(-1)^n\det(C)$, we have that, by \eqref{2.5}
\begin{equation}\label{2.16A}
  \det(A) = (-1)^{n+1}\bar{\alpha}_{n-1}
\end{equation}
\end{theorem}

This implies that if $A^\mu$ is unitarily equivalent to $A^\nu$, then $\mu$ and $\nu$ have the same $\Phi_n$ and so, by Wendroff's Theorem, the same $\{\alpha_j\}_{j=0}^{n-1}.$

Suppose now that $d\mu$ is a trivial measure on $\partial\bbD$, say with $n+1$ pure points, $\{w_j\}_{j=1}^{n+1}$. Then $\{z^k\}_{k=0}^{n}$ are still independent, so one can use Gram--Schmidt to form $\{\Phi_j\}_{j=0}^n$.  As the norm minimizer, one also has that
\begin{equation}\label{2.17}
  \Phi_{n+1}=\prod_{j=1}^{n+1}(z-w_j)
\end{equation}
Since this has norm $0$, one expects (see \eqref{2.6}) and indeed finds that $\Phi_{n+1}$ is given by Szeg\H{o} recursion but with $|\alpha_n|=1$.  That is, trivial measures are described by sets of $n+1$ Verblunsky coefficients, the first $n$ in $\bbD$ and the last in $\partial\bbD$.   The corresponding multiplication operators are precisely the $n+1$--dimensional unitaries with a cyclic vector.  Moreover, the Schur and Geronimus Theorems extend. The Schur function of such an $n+1$ point measure is, up to a leading phase factor, an $n$--fold Blaschke product (whose zeros are those of $\Phi_n$) and which has $n+1$ Schur parameters, $n$ in $\bbD$ and the last in $\partial\bbD$.

This motivates, the following: Suppose, we are given a non--trivial measure with Verblunsky coefficients, $\{\alpha_j\}_{j=0}^\infty$, and we consider, $\Phi_n(z)$.  Given $\lambda\in\partial\bbD$, we define the \emph{paraorthogonal polynomial} (POPUC) by
\begin{equation}\label{2.18}
  \Phi_{n+1}(z;\lambda) = z\Phi_n(z) - \bar{\lambda} \Phi_n^*(z)
\end{equation}
The definition goes back to Delsarte--Genin \cite{DG} and Jones et. al \cite{Jones}; among later papers, we mention \cite{CMV1,CMV2,Cast1,Cast2,Darius,Golin,SimanekPOPUC,SimonPOPUC,Wong}.  One can show that the $n+1$ point measure, $d\nu_\lambda$, whose first $n$ Verblunsky coefficients are the first $n$ $\alpha_j$ and with $\alpha_n=\lambda$ has $\Phi_{n+1}(z;\lambda)$ as its $n+1$-st monic OPUC (which has norm 0!).  We'll use $U_\lambda$ as multiplication by $z$ on $L^2(\partial\bbD,d\nu_\lambda)$. This $L^2$ space has dimension $n+1$ and can be viewed as the space of polynomials of degree at most $n$.  It will be useful to sometimes view it as an $n+1 \times n+1$ matrix in some convenient basis -- most often the normalized OPUC but sometimes the eigenvectors.

\begin{theorem} \lb{T2.10} Fix $\{\alpha_j\}_{j=0}^{n-1}$ all in $\bbD$ and let $A$ be the corresponding compressed multiplication operator. The POPUC of degree $n+1$ are in one to one correspondence with the rank one unitary dilations of $A$. The eigenvalues of the unitary, $U_\lambda$, associated to $\Phi_{n+1}(z;\lambda)$ are the zeros of that polynomial so that
\begin{equation}\label{2.19}
  \det(z-U_\lambda) = \Phi_{n+1}(z;\lambda)
\end{equation}
In particular,
\begin{equation}\label{2.19A}
  \det(U_\lambda) = (-1)^n \, \overline \lambda
\end{equation}
\end{theorem}

We note that since $\Phi_n^*(0)=1$, the constant term in $\Phi_{n+1}(z)$ is $-\bar{\lambda}$ so we have that
\begin{equation}\label{2.20}
  \lambda = (-1)^n \prod_{j=0}^{n+1}\overline{w_j}
\end{equation}
and we note that
\begin{equation}\label{2.21}
  \Phi_{n+1}^*(z,\lambda) = -\lambda \Phi_{n+1}(z,\lambda)
\end{equation}

We will need an explicit matrix representation for $M_z$ and so $A^\mu$ which \cite{OPUC1} calls the GGT representation after Geroniums \cite{Geron1944}, Gragg \cite{Gragg} and Teplyaev \cite{Tep}.  We note there is another matrix representation called the CMV representation (after \cite{CMV}) discussed in \cite[Sectioon 4.2]{OPUC1} which has advantages for the study of $M_z$ on its infinite dimensional space but for our needs on the $n\times n$ operator, $A^\mu$, the GGT representation is simpler and more natural.

$\{\varphi_j\}_{j=0}^{n-1}$ is an orthonormal basis for $\calP_{n-1}$ and the matrix elements of $A^\mu$ in that basis are
\begin{equation}\label{2.22}
  \calG_{k\ell} = \jap{\varphi_k, z \varphi_\ell}; \qquad k,\ell=0,\dots,n-1
\end{equation}
Their explicit calculation is (see \cite[(4.15)]{OPUC1}) with $\alpha_{-1}=-1$
\begin{equation} \lb{2.23}
\calG_{k\ell}=\left \{ \begin{array}{cc}
-\overline{\alpha}_\ell \alpha_{k-1}\prod_{j=k}^{\ell-1} \rho_{j} & 0 \leq k \leq \ell \\
\rho_\ell & k=\ell+1 \\
0 & k\geq \ell+2
\end{array} \right.
\end{equation}

In \cite{OPUC1}, this is calculated using $\jap{\Phi_n^*,P} = \norm{\Phi_n}^2 P(0)$ if $\deg P \le n$.  An easier alternative taken from \cite{BSZ2} is to use Szeg\H{o} recursion and inverse Szeg\H{o} recursion in the form
\begin{align}
  z\varphi_n(z) &= \rho_n \varphi_{n+1}(z) + \bar{\alpha}_n \varphi_n^*(z) \lb{2.24} \\
  \varphi_j^*(z) &= \rho_{j-1} \varphi_{j-1}^*(z) - \alpha_{j-1} \varphi_j(z) \lb{2.25}
\end{align}
so
\begin{align}
  z\varphi_n(z) &= \rho_n \varphi_{n+1}(z) - \bar{\alpha}_n \alpha_{n-1} \varphi_n(z) + \bar{\alpha}_n \rho_{n-1} \varphi_{n-1}^*(z) \nonumber \\
                &= \rho_n \varphi_{n+1}(z) - \bar{\alpha}_n \alpha_{n-1} \varphi_n(z) - \bar{\alpha}_n \rho_{n-1} \alpha_{n-2} \varphi_{n-1}(z) \nonumber \\
                &\null \qquad \qquad + \bar{\alpha}_n \rho_{n-1} \rho_{n-2} \varphi_{n-2}^*(z) \lb{2.26}
\end{align}
which upon iterating (using that $\varphi_0^*=1=\varphi_0$) yields
\begin{equation}\label{2.27}
  z\varphi_n(z) = \rho_n \varphi_{n+1}(z) + \sum_{k=0}^{n} \calG_{kn} \varphi_k(z)
\end{equation}
with $\calG$ given by \eqref{2.23}.  In other words,
\begin{equation}\label{2.28}
  \calG(\{\alpha_j\}_{j=0}^{n-1})= \left(
               \begin{array}{ccccc}
                  \bar{\alpha}_0 & \bar{\alpha}_1\rho_0 & \bar{\alpha}_2 \rho_0\rho_1 & \dots & \bar{\alpha}_{n-1}\rho_0\dots\rho_{n-2} \\
                  \rho_0 & -\bar{\alpha}_1\alpha_0 & -\bar{\alpha}_2\alpha_0\rho_1 & \dots &  -\bar{\alpha}_{n-1}\alpha_0\rho_1\dots\rho_{n-2}\\
                  0 & \rho_1 & -\bar{\alpha}_2\alpha_1 & \dots &  -\bar{\alpha}_{n-1}\alpha_1\rho_2\dots\rho_{n-2} \\
                  \vdots & \vdots & \vdots & \ddots & \vdots \\
                  0 & 0 & 0 & \dots & -\bar{\alpha}_{n-1}\alpha_{n-2} \\
               \end{array}
                                   \right)
\end{equation}

We call the GGT representation of a compressed multiplication operator a \emph{GGT contraction}.  The GGT representation of the matrix associated to a POPUC we'll call a \emph{GGT unitary}.

%%%%%%%%%%%%%%%%%%%%%%%%%%%%%%%%%%%%%%%%%%%%%%%%%%%%%%%%%%%%%%
\section{Main Results} \lb{s3}
%%%%%%%%%%%%%%%%%%%%%%%%%%%%%%%%%%%%%%%%%%%%%%%%%%%%%%%%%%%%%%

\renewcommand\thetheorem{\arabic{theorem}}

Here are the main results of this paper.  We begin with three theorems that capture the main structure theorem for $S_n$ (Theorem \ref{TD}).

\begin{theorem} \lb{T1} Every compressed multiplication operator lies in $S_n$.
\end{theorem}

\begin{theorem} \lb{T2} Every element in $S_n$ is unitarily equivalent to a compressed multiplication operator.
\end{theorem}

\begin{theorem} \lb{T3} For any set of $n$ elements (with multiplicity) in $\bbD$, there is a compressed multiplication operator with those eigenvalues.  Two compressed multiplication operators with the same characteristic polynomial are unitarily equivalent.
\end{theorem}

Thus, we can parametrize equivalence classes of $S_n$ by two sets of $n$ elements from $\bbD$:
\begin{enumerate}
	\item[(i)] eigenvalues of a representative $A\in S_n$ of the equivalence class, or
	\item[(ii)] Verblunsky coefficients $\{\alpha_j\}_{j=0}^{n-1}$, associated to the compressed multiplication operator.
\end{enumerate}

Theorem \ref{T1} will be easy and we'll see that given the first two theorems, Theorem \ref{T3} is essentially a restatement of Wendroff's Theorem. Thus the key will be Theorem \ref{T2}.  We will have two proofs, neither so short.  The one in Section \ref{s5} will rely on turning a GGT unitary on its head, i.e. noting that a GGT unitary flipped along the secondary diagonal is again a GGT unitary but with different Verblunsky coefficients. We'll develop this idea in Section \ref{s4}.  Our second proof of Theorem \ref{T2} in Section \ref{s6} will involve constructing an orthonormal basis for the space on which $A \in S_n$ acts that will show it is acting as a compressed multiplication operator.  Our construction of this basis is motivated by inverse Szeg\H{o} recursion.

Arlinski\v{i} has four papers \cite{Arl1,Arl2,Arl3,Arl4}, one with coauthors, that are related to the above three theorem.  These papers deal with a related, but different, class of operators which he calls truncated CMV matrices. He gets these by starting with a finite or infinite unitary CMV matrix and stripping off their first row and column.  If the original unitary is a finite $n+1\times n+1$ matrix, this is related to, but distinct from, our compressed multiplication operators in CMV basis.  It is cleaner if one describes the difference in GGT basis.  Then he gets his operators by removing the first row and column and we remove the last row and column.  For his objects, he has analogs of Theorems \ref{T1} and \ref{T2}.  He also has a theorem like Theorem \ref{T3} except that, because truncated CMV matrices whose Verblunsky coefficients are related up to an overall phase factor are unitarily equivalent, the uniqueness result is a little more involved.  Because of our Theorem \ref{T4.2}, one can establish a simple unitary equivalence between truncated CMV matrices and compressed multiplication operators and then obtain our Theorems \ref{T1}-\ref{T3} from his theorems or vice--versa.  We note that although he doesn't seem to know about the earlier work around Theorem \ref{TD}, like us, he has an OPUC proof of it.

Next, in Section \ref{s7}, we'll turn to the study of the numerical range of compressed multiplication operators and prove two theorems.

For each $\lambda\in\partial\bbD$, let $U_\lambda$ be the associated unitary, $\Phi_{n+1}$ the associated POPUC, $\{w_j\}_{j=1}^{n+1}$ the zeros of $\Phi_{n+1}$, cyclically ordered, $\eta_j$ the associated normalized eigenvectors, so $\eta_j(z) = N_j^{-1} \Phi_{n+1}(z)/(z-w_j)$ and $d\mu$ the spectral measure
\begin{equation}\label{defmu}
d\mu= \sum_{j=1}^{n+1} |\jap{\eta_j,1}|^2\delta_{w_j}
\end{equation}
Let $m_j=|\jap{\eta_j,\varphi_{n}}|^2>0$ (since $\deg(\eta_j)=n$) so $\sum_{j=1}^{n+1} m_j = 1$ (since $\norm{\varphi_n}=1$).  Let $A$ be the dimension $n$ compressed multiplication operator and $w_{n+2}\equiv w_1$.

\begin{theorem} \lb{T4} For $j=1,\dots,n+1$, the line from $w_j$ to $w_{j+1}$ intersects $N(A)$ in a single point, $\zeta_j$, and $|\zeta_j-w_j|/|\zeta_j-w_{j+1}|= m_j/m_{j+1}$.  In particular, $\prod_{j=1}^{n+1}|\zeta_j-w_j| = \prod_{j=1}^{n+1} |\zeta_j-w_{j+1}|$.
\end{theorem}

\begin{remark} If the line is between $w_j$ and $w_k$, the corresponding point also lies in $N(A)$ but it is not, in general, the only point in $N(A)$.
\end{remark}

\begin{theorem} \lb{T5} For each $\lambda$, we have that $N(U_\lambda)$ is a solid $(n+1)$-gon whose sides are tangent to $N(A)$. $\partial N(A)$ is a strictly convex analytic curve and one has that
\begin{equation}\label{3.1}
  N(A)=\cap_{\lambda\in\partial\bbD}N(U_\lambda)
\end{equation}
\end{theorem}

The proofs, which rely on operator theory, will not be difficult.

With the above definition of the $m_j$, the spectral measure for $\varphi_n$ and the operator $U_\lambda$ is
\begin{equation}\label{defnu}
d\nu=\sum_{j=1}^{n+1} m_j\, \delta_{w_j}
\end{equation}

\begin{theorem} \lb{T6} Fix a degree $n+1$ POPUC and let $d\nu$ be the measure constructed above.  Then one has that
\begin{equation}\label{3.2}
 \int \frac{1}{z-e^{i\theta}} \, d\nu(\theta) = \frac{\Phi_n(z)}{z\Phi_n(z)-\bar{\lambda}\Phi_n^*(z)}=\frac{\Phi_n(z)}{\Phi_{n+1}(z,\lambda)}
\end{equation}
Conversely, let $d\nu$ be an $n+1$ point probability measure on $\partial\bbD$.  Then \eqref{3.2} holds for the POPUC defined by $d\nu$.
\end{theorem}

If $\{z_j\}_{j=1}^n$ are the zeros of $\Phi_n$, then $\Phi_n(z)/\Phi_n^*(z) = \prod_{j=1}^{n}\tfrac{z-z_j}{1-\bar{z}_jz}$, so Theorem \ref{T6} is the Blaschke product theorem of Daepp et al \eqref{1.3} and its converse as discussed after \eqref{1.4} with a very different proof of the positivity of the $m_j$ and of the formula.  As we'll discuss this result can be proven from Khrushchev's formula by taking limits to extend his formula to trivial measures.  It can also be obtained from general formulae for $M$--functions in \cite{OPUC1}.  We will give two simple direct proofs.  From the OPUC point of view, the picture isn't complete until we find the Verblunsky coefficients for the trivial measure $d\nu$.

\begin{theorem} \lb{T7} Let $\{\alpha_j\}_{j=0}^{n-1}$ be the Verblunsky coefficients for the spectral measure, $d\mu$, of \eqref{defmu}, that is the measure for which $\varphi_j$ are the OPUC. Then the Verblunsky coefficients $\{\alpha_j(d\nu)  \}_{j=0}^{n-1}$ for the measure, $d\nu$\, of \eqref{defnu} are given by
\begin{equation}\label{3.3}
  \alpha_{j}(d\nu) = -\lambda \bar{\alpha}_{n-1-j},\quad j=0,\dots,n-1; \qquad \alpha_{n}(d\nu)=\lambda
\end{equation}
\end{theorem}
The $\lambda=1$ case of this result is implicit in a remark on rank two decoupling of CMV matrices in \cite{OPUC1} but we'll give two more direct proofs: one using the results from Section \ref{s4} on turning a CMV matrix on its head and one proof using Geronimus' Theorem and Szeg\H{o} recursion.  Theorems \ref{T6} and \ref{T7} will be proven in Sections \ref{s8} and \ref{s9}.

Section \ref{s10} has two Wendroff type theorems (Section \ref{s10} will discuss previous literature related to this result but we note here that Theorem \ref{T8} appears previously in \cite{GolKud}.)

\begin{theorem} [Wendroff's Theorem for (P)OPUC] \lb{T8} The zero's of POPUC's for two values of $\lambda$ interlace.  Conversely, given two sets of $n+1$ interlacing points on $\partial\bbD$, there exist unique $\{\alpha_j\}_{j=0}^{n-1}$ in $\bbD$ and $\lambda, \mu$ in $\partial\bbD$ so these are zeros of the associated POPUCs.
\end{theorem}

We put the (P) in ``(P)OPUC'' because, as we'll discuss, there is a rather different Wendroff type theorem for POPUC in the literature which only involves some POPUC parameters whereas we discuss both the POPUC parameter and the OPUC parameters.

\begin{theorem} [Wendroff's Theorem for Second Kind POPUC] \lb{T9} Let $\{w_j\}_{j=1}^{n+1}$ be the zeros of a POPUC, $\Phi_{n+1}$, ordered clockwise, and $\{y_j\}_{j=1}^{n+1}$ be the zeros of associated second kind POPUC, $\Psi_{n+1}$, also ordered clockwise, so that $y_1$ is the first zero after $w_1$ going clockwise.  Then the $w$'s and $y$'s strictly interlace and one has that
\begin{equation}\label{3.4}
  \prod_{j=1}^{n+1} y_j = - \prod_{j=1}^{n+1} w_j
\end{equation}
Conversely, if $\{w_j\}_{j=1}^{n+1}$ and $\{y_j\}_{j=1}^{n+1}$ are strictly interlacing and obey \eqref{3.4}, then there is a unique set of Verblunsky coefficients $\alpha_0,\dots,\alpha_{n-1} \in \bbD$ and $\lambda\in\partial\bbD$ so that $\{w_j\}_{j=1}^{n+1}$ is the set of zeros of the associated POPUC and $\{y_j\}_{j=1}^{n+1}$  the zeros of the associated second kind POPUC.
\end{theorem}

Recall that for OPUC, second kind OPUC were defined by putting a minus sign in front of all the Verblunsky coefficients.  We define the second kind POPUC by also putting a minus sign in front of the $\lambda$ parameter unlike the convention in \cite{SimonPOPUC}.  Theorem \ref{T8} is equivalent to the result of Gau--Wu quoted as Theorem \ref{TC}.  We have a new proof.  The first halves of both of our Wendroff type theorems are already in the OPUC literature; we'll say more about that in Section \ref{s10}.

Section \ref{s11} discusses a theorem of Gau-Wu \cite{GW1999} about zeros of derivatives of polynomials all of whose zeros are on $\partial\bbD$ and the question of Gorkin--Skubak \cite{Gorkin2011} asking when a finite subset of $\bbD$ is the set of zeros of the derivative of a polynomial all of whose zeros lie on $\partial\bbD$.

Section \ref{s12} discusses Poncelet's Theorem.  Theorem \ref{TA} shows that the eigenvalues of $A\in S_n$ are the foci of a curve containing set $\partial N(A)$.  While we don't have an OPUC way to understand the eigenvalues as foci, Section \ref{s13} explains two other ways of going from $N(A)$ to the eigenvalues.

%%%%%%%%%%%%%%%%%%%%%%%%%%%%%%%%%%%%%%%%%%%%%%%%%%%%%%%%%%%%%%
\section{Turning a Unitary GGT Matrix on Its Head} \lb{s4}
%%%%%%%%%%%%%%%%%%%%%%%%%%%%%%%%%%%%%%%%%%%%%%%%%%%%%%%%%%%%%%

\renewcommand\thetheorem{\arabic{theorem}}
\numberwithin{theorem}{section}

In this section we answer a question about unitary $n\times n$ GGT matrices that will be relevant to one of our proofs of Theorem \ref{T2} and \ref{T7} (although the later will need $(n+1)\times (n+1)$ matrices).  We begin with

\begin{theorem} \lb{T4.1} Let $U$ be an $n\times n$ unitary matrix with cyclic vector $\psi_0$.  Then there is a unique basis $\{e_j\}_{j=1}^n$ in which $U$ is a GGT unitary $\calG(\{\alpha_j\}_{j=0}^{n-1})$ and $\psi_0=e_1$.  The $\{\alpha_j\}_{j=0}^{n-1}$ are uniquely determined by the pair $U, \psi_0$.
\end{theorem}

\begin{remarks} 1. We emphasize that the $\{\alpha_j\}_{j=0}^{n-1}$ are not unitary invariants of $U$ alone but of $U$ and its cyclic vector.  The unitary invariants of $U$ are the $n$ eigenvalues so $n$ real parameters, but the $\alpha's$ are $n-1$ in $\bbD$ and one in $\partial\bbD$ (since $\alpha_{n-1}\in\partial\bbD$) and so $2n-1$ real parameters.  To determine $\calG$, all that matters from the cyclic vector is the spectral measure, so $n-1$ independent real weights leading to the correct parameter count.  That the eigenvalues of $U$ aren't enough to determine the $\alpha$'s is indicated by the form of the Wendroff type theorems in Section \ref{s10}.

2. While we won't need it below, there are explicit formulae for the weights of the spectral measure $d\mu = \sum_{j=1}^{n} q_j \delta_{w_j}$ for $U$ and $\psi_0$ in terms of the OPUC $\{\varphi_j\}_{j=0}^{n-1}$ (which are determined by $\{\alpha_j\}_{j=0}^{n-2}$).  Namely $q_j=1/\lambda_{n-1}(w_j);\,\lambda_{n-1}(w)=\sum_{j=0}^{n-1}|\varphi_j(w)|^2$.  The $\lambda$ are called Christoffel numbers.  This formula goes back to Jones et al. \cite{Jones} (see \cite[Theorem 2.2.12]{OPUC1} for a quick proof) although it is just the POPUC analog of well--known formulae for OPRL known under the name Gauss quadrature or Gauss--Jacobi quadrature.
\end{remarks}

\begin{proof} The pair $U,\psi_0$ determine the spectral measure and so a set of $n+1$ monic OPUC whose last element is a POPUC and whose first $n$ elements normalized are the basis obtained by using Gram--Schmidt on $\{U^k\psi\}_{k=0}^{n-1}$.  The representation in this matrix is precisely the claimed GGT matrix.  Uniqueness is easy.
\end{proof}

Here is the main result of this section:

\begin{theorem} \lb{T4.2} Let $U$ be an $n\times n$ unitary matrix with cyclic vector $\psi_0$ with $\det(U) = (-1)^{n+1}$.  Then there is a unique basis $\{e_j\}_{j=1}^n$ in which $U$ is a GGT unitary $\calG(\{\beta_j\}_{j=0}^{n-1})$ and $\psi_0=e_n$.  The $\{\beta_j\}_{j=0}^{n-1}$ are uniquely determined by the pair $U, \psi_0$ and the $\beta$'s are related to the $\alpha$'s of Theorem \ref{T4.1} by
\begin{equation}\label{4.1}
  \beta_j = \bar{\alpha}_{n-2-j}, \quad j=0,\dots,n-2; \qquad \beta_{n-1} = -1
\end{equation}
\end{theorem}

\begin{remarks}  1.  At first sight, it may be puzzling that $U$ is unitarily equivalent to GGT matrices with, in general, two different Verblunsky coefficients.  After all, Wendroff's Theorem for OPUC implies that a GGT contraction determines uniquely all its Verblunsky coefficients but as pointed out in Remark 1 to Theorem \ref{T4.1} this is not true for GGT unitaries as this theorem dramatically demonstrates.

2.  The intuition is simple.  A look at \eqref{2.28} shows that when $\alpha_{n-1}=-1$, there is a covariance under reflection about the reverse main diagonal so long as one relabels the $\alpha$'s via \eqref{4.1}.
\end{remarks}

The covariance under reflection suggests we want to reverse the order of the GGT basis.  That has the advantage of taking $e_1$ to $e_n$ as we want but such a reversal not only reflects, it also gives a transpose so the following is critical

\begin{lemma} \lb{L4.3} Given any set of Verblunsky coefficients  $\{\alpha_j\}_{j=0}^{n-1}$, with  $\alpha_j\in\bbD, j=0,\dots,n-2$, and  $\alpha_{n-1}\in\overline{\bbD}$, the matrix representation of $A$, the associated compressed multiplication operator, in the basis $\{\chi_j\}_{j=0}^{n-1}$ where $\chi_j = \overline{\varphi}_j$ is the transpose of the GGT matrix.
\end{lemma}

\begin{remark} It is perhaps surprising a finite GGT matrix is unitarily equivalent to its transpose.  Using the so called AGR factorization (see \cite{CMV5years}) one can write the unitary operator explicitly.  If $U_1\dots U_n$ is the AGR factorization and $W_j = U_j\dots U_n$ and $W=W_n\dots W_2$, then one can show that $W\calG W^{-1} = \calG^t$.
\end{remark}

\begin{proof} Clearly
\begin{align}
  \jap{\chi_j,z\chi_k} &= \int \bar{\chi}_j(e^{i\theta})e^{i\theta}\chi_k(e^{i\theta})\,d\mu(\theta) \nonumber \\
                       &= \int \varphi_j(e^{i\theta})e^{i\theta}\bar{\varphi}_k(e^{i\theta})\,d\mu(\theta) \nonumber  \\
                       &= \calG_{kj}\label{4.2}
\end{align}

\end{proof}

\begin{proof} [Proof of Theorem \ref{T4.2}]  As in the proof of Theorem \ref{T4.1}, realize the space on which $U$ acts as $L^2(\partial\bbD,d\mu)$ where $d\mu$ is the spectral measure for $\psi_0$ and $U$.  Let $\varphi_j, j=0,\dots,n-1$ be the orthonormal polynomials for $d\mu$, so, in particular, $\psi_0=\varphi_0$.  Let $e_j=\bar{\varphi}_{n-j}$ so $e_n=\psi_0$.  By either the explicit formulae, \eqref{2.23} for $\calG_{k\ell}$ or the form of the matrix \eqref{2.28}, one sees that the transpose of $\calG(\{\alpha_j\}_{j=0}^{n-1})$ in a basis run backwards is $\calG(\{\beta_j\}_{j=0}^{n-1})$ with $\beta$ given by \eqref{4.1}.  Given Lemma \ref{L4.3}, we see that the matrix of $U$ in basis $\{e_j\}_{j=1}^n$ is $\calG(\{\beta_j\}_{j=0}^{n-1})$ as claimed.
\end{proof}

We supposed that $\det(U)=(-1)^{n+1}$ for simplicity of calculation (which implies that $\alpha_{n-1}=-1$).  It is easy to see that, more generally, one has that

\begin{theorem} \lb{T4.4} Let $U$ be an $n\times n$ unitary matrix with cyclic vector $\psi_0$.  Then there is a unique basis $\{e_j\}_{j=1}^n$ in which $U$ is a GGT unitary $\calG(\{\beta_j\}_{j=0}^{n-1})$ and $\psi_0=e_n$.  The $\{\beta_j\}_{j=0}^{n-1}$ are uniquely determined by the pair $U, \psi_0$ and the $\beta$'s are related to the $\alpha$'s of Theorem \ref{T4.2} by
\begin{equation}\label{4.3}
  \beta_j = -\alpha_{n-1}\bar{\alpha}_{n-2-j}, j=0,\dots,n-2; \qquad \beta_{n-1} = \alpha_{n-1}
\end{equation}
\end{theorem}

%%%%%%%%%%%%%%%%%%%%%%%%%%%%%%%%%%%%%%%%%%%%%%%%%%%%%%%%%%%%%%
\section{$S_n$ = Compressed Multiplication via GGT Matrices} \lb{s5}
%%%%%%%%%%%%%%%%%%%%%%%%%%%%%%%%%%%%%%%%%%%%%%%%%%%%%%%%%%%%%%

In this Section, we give our first proof of Theorem \ref{T2}. Let $A \in S_n$. Pick once and for all a unit vector $x_{n-1}\in \ran(\bdone-A ^*A )$.  Then there exists some $\rho_{n-1}\in (0,1]$ such that
\begin{equation}\label{5.1}
   A^*A=\bdone-\rho_{n-1}^2\jap{x_{n-1},\cdot}x_{n-1}
\end{equation}
(for now, $ \rho_{n-1}$ is just some number in $(0,1]$ but eventually we will see that it is the same as a $\rho_{n-1}$ associated to OPUC and defined in \eqref{2.4}).
By using the polar decomposition, we can write $A=UB$ where $U$ is unitary with $\det(U)=(-1)^{n+1}$ and $B$ is diagonal in a basis including $x_{n-1}$ as its last element so that $B$ has $1$'s along diagonal except for the last diagonal element which is some complex number $a$ with $|a|^2+\rho_{n-1}^2=1$.  To see this start with $A=V|A|$ as polar decomposition, and shift to a basis where $|A|$ is diagonal.  Then pick $U$=$VD$ where $D$ is diagonal and unitary with ones along diagonal except for the last slot whose phase is picked to arrange $\det(U)=(-1)^{n+1}$ and then take $B=D^*|A|$.

\begin{lemma} \lb{L5.1} For any $A$ obeying \eqref{5.1}, with the above decomposition  $A=UB$, we have that $x_{n-1}$ is cyclic for $U$ if and only if $A$ has no eigenvalue of magnitude $1$.
\end{lemma}
\begin{proof} 	If $x_{n-1}$ was not cyclic for $U$, let $\calK$ be the orthogonal complement of the cyclic subspace generated by $U$ and $x_{n-1}$. Then $U^{-1}$ is a polynomial in $U$ (look at the secular equation) so $U^*$ leaves $\calK^\perp$ invariant and thus $U$ leaves $\calK$ invariant.  Since $\calK$ is orthogonal to $x_{n-1}$, we have that $B\restriction \calK = \bdone$.  It follows that any eigenvector of $U\restriction\calK$ is an eigenvector of $A$ with the same eigenvalue so $A$ has eigenvalues of magnitude $1$ as claimed.

Conversely, if $A\eta=\lambda\eta$ with $\norm{\eta}=|\lambda|=1$, then $\norm{A\eta}=\norm{\eta}$ so $\eta$ is orthogonal to $x_{n-1}$, implying that $B\eta = \eta$, and is thus also an eigenvector of $U$.  It follows that $\jap{\eta,U^kx_{n-1}}=\jap{(U^*)^k\eta,x_{n-1}}=0$ so $\eta$ is orthogonal to the cyclic subspace of $U$ and $x_{n-1}$ so $x_{n-1}$ is not cyclic for $U$.
\end{proof}

\begin{proof} [First Proof of Theorem \ref{T2}] Let $A\in S_n$.  Then, as above $A=UB$ where, by Lemma \ref{L5.1}, $U$ has $x_{n-1}$ as cyclic vector.  By Theorem \ref{T4.2}, $U$ is a GGT unitary where $\beta_{n-1} = -1$ and $B$ is diagonal with $1$'s along the diagonal except for $a$ in the lower right corner. It follows that in this basis $A$ is a GGT contraction with Verblunsky coefficients
\begin{equation}\label{5.2}
  \gamma_j = \left\{
               \begin{array}{ll}
                 \beta_j, & \hbox{ if } 0\le j\le n-2\\
                 -a, & \hbox{ if } j=n-1
               \end{array}
             \right.
\end{equation}
Thus $A$ is a compressed multiplication operator.
\end{proof}

%%%%%%%%%%%%%%%%%%%%%%%%%%%%%%%%%%%%%%%%%%%%%%%%%%%%%%%%%%%%%%
\section{$S_n$ = Compressed Multiplication via Inverse Szeg\H{o} Recursion} \lb{s6}
%%%%%%%%%%%%%%%%%%%%%%%%%%%%%%%%%%%%%%%%%%%%%%%%%%%%%%%%%%%%%%

In this Section, we give our second proof of Theorem \ref{T2}. Let $A \in S_n$.  We will carefully pick a basis $\{x_j\}_{j=0}^{n-1}$ with $x_{n-1}$ a unit vector as in Section \ref{s5}, for which
\begin{equation}\label{6.1}
 % A^*A=\bdone-\rho_{n-1}^2\jap{x_{n-1},\cdot}x_{n-1}
 A^*A=\bdone-\rho_{n-1}^2 x_{n-1}  x_{n-1}^*
\end{equation}
and show that this basis is the basis of OPUC for a compressed multiplication operator.  Our motivation for the construction is inverse Szeg\H{o} recursion.

In \eqref{6.1}, the   unit vector $x_{n-1}\in\calH_n$, the space on which $A$, which we'll now denote by $A_n$, acts, so $x_{n-1} \in\ran(\bdone-A_n^*A_n)$, and $\rho_{n-1}\in (0,1]$.

Suppose that $n\ge 2$ (Theorem \ref{T2} is trivial when $n=1$).
 Let $\calH_{n-1}$ be the $n-1$--dimensional orthogonal complement of $x_{n-1}$ so
\begin{equation}\label{6.2}
  \calH_{n-1}=\ker(\bdone-A_n^*A_n)
\end{equation}
and let $P_{n-1}$ be the projection on $\calH_{n-1}$.

Since $A_n$ is an isometry on $\calH_{n-1}$, $A_n[\calH_{n-1}]$, the image of $\calH_{n-1}$ under $A_n$ has dimension $n-1$ so we can pick a unit vector $y_{n-1}\in A_n[\calH_{n-1}]^\perp$ unique up to overall phase factor.  For now we won't specify that phase but we will later (essentially so that if $x_{n-1}$ is $\varphi_{n-1}$, then $y_{n-1}$ is $\varphi_{n-1}^*$).  Let $Q_{n-1}$ be the projection onto $A_n[\calH_{n-1}]$.

\begin{proposition} \lb{P6.1} First, $Q_{n-1}x_{n-1} \ne 0$. Secondly there exists $a_{n-1}\in\bbC$ so that
\begin{equation}\label{6.2a}
  \rho_{n-1}^2+|a_{n-1}|^2=1
\end{equation}
\begin{equation}\label{6.2b}
  A_nx_{n-1} = a_{n-1}y_{n-1}
\end{equation}
Thirdly, there exists a unique unit vector $x_{n-2} \in \calH_{n-1}$ so that $A_nx_{n-2}$ is a positive multiple of $Q_{n-1}x_{n-1}$.  One has that
\begin{equation}\label{6.3}
  A_nx_{n-2} = \rho_{n-2}^{-1}[x_{n-1} + a_{n-2} y_{n-1}]
\end{equation}
for some $\rho_{n-2} \in (0,1]$ and $a_{n-2}\in\bbC$ where
\begin{equation}\label{6.4}
  \rho_{n-2}^2+|a_{n-2}|^2 = 1
\end{equation}
Let $A_{n-1}=P_{n-1}A_{n}P_{n-1}$ as an operator on $\calH_{n-1}$.  Then $A_{n-1} \in S_{n-1}$ with
\begin{equation}\label{6.5}
  \bdone-A_{n-1}^*A_{n-1} = \rho_{n-2}^2\jap{x_{n-2},\cdot}x_{n-2}
\end{equation}
\end{proposition}

\begin{proof}
If $Q_{n-1}x_{n-1}=0$, then $A_n$ leaves $\calH_{n-1}$ invariant.  Since $A_n$ is an isometry on that space, the restriction of $A_n$ to that invariant subspace would be unitary which means that $A_n$ would have eigenvalues of magnitude one contradicting the hypothesis that $A \in S_n$ (it is interesting that this is how the eigenvalue condition enters in the proof!).

By \eqref{6.1}, if $u\in\calH_{n-1}$, then $\jap{A_nu,A_nx_{n-1}}=\jap{u,x_{n-1}}=0$ so $A_nx_{n-1}\in A_n[\calH_{n-1}]^\perp$ proving \eqref{6.2b}.  Formula \eqref{6.1}, which says that $\norm{A_nx_{n-1}}^2=1-\rho_{n-1}^2$ then implies \eqref{6.2a}.

Since is $A_n$ is an isometry on $\calH_{n-1}$, there is a unique vector $u \in \calH_{n-1}$ so that $A_nu = Q_{n-1}x_{n-1}$ and one has that $\rho_{n-2} \equiv \norm{u} \in (0,1]$.  Let $x_{n-2}=\rho_{n-2}^{-1}u$ and note that $(\bdone-Q_{n-1})x_{n-1} = -a_{n-2}y_{n-1}$ for some $a_{n-2}\in\bbC$.  This proves \eqref{6.3}.  Writing \eqref{6.3} as
\begin{equation}\label{6.6}
  x_{n-1} = \rho_{n-2}A_nx_{n-2}-a_{n-2}y_{n-1}
\end{equation}
and noting that $y_{n-1}$ is orthogonal to $Ax_{n-2}$ implies \eqref{6.4}.

Since $A_n$ is an isometry on $\calH_{n-1}$, it maps the orthogonal complement of $u$ in $\calH_{n-1}$ into $\calH_{n-1}$ so $A_{n-1}$ is an isometry on that orthogonal complement which proves \eqref{6.5}.  Moreover, if $A_{n-1}v=cv$ with $v$ a unit vector and $|c|=1$, then $\norm{P_{n-1}A_nv}=1 \ge \norm{A_nv}$ so $P_{n-1}A_nv=A_nv$ and $v$ is an eigenvector of $A_n$ and there is a contradiction, i.e. $A_{n-1}\in S_{n-1}$.
\end{proof}

If $n\ge 2$, we define, consistently with viewing $A_{n-1}\in S_{n-1}$, $\calH_{n-2}$ to be the complement of $x_{n-2}$ in $\calH_{n-1}$, i.e.
\begin{equation}\label{6.7}
  \calH_{n-2} = \ker(\bdone-A_{n-1}^*A_{n-1})
\end{equation}
and $Q_{n-2}$, an operator on $\calH_{n-1}$, the projection onto $A_{n-1}[\calH_{n-2}]$.  If $n=2$, then we take $Q_0=0$.

\begin{proposition} \lb{P6.2} One can choose a unit vector $y_{n-2}\in\calH_{n-1}$ with $Q_{n-2}y_{n-2}=0$ (i.e. $y_{n-2}\in\calH_{n-1}\cap [A_{n-1}[\calH_{n-2}]]^\perp)$ so that
\begin{equation}\label{6.8}
  Ax_{n-2} = \rho_{n-2}x_{n-1}+a_{n-2}y_{n-2}
\end{equation}

\begin{equation}\label{6.9}
  \jap{y_{n-1},y_{n-2}} = \rho_{n-2} > 0
\end{equation}
\end{proposition}
\begin{remarks} 1. \eqref{6.8} is, of course, (direct) Szeg\H{o} recursion \eqref{2.3} and, as noted, \eqref{6.3} is inverse Szeg\H{o} recursion, \eqref{2.7}. We note that \eqref{6.9} is motivated by  \cite[(1.5.61)]{OPUC1} while \eqref{6.10} below is its $^*$ dual and \eqref{6.13} below is a special case of \cite[(1.5.60)]{OPUC1}.

2. This also holds when $n=2$.  $\calH_{n-1}$ is then one dimensional so $y_{n-1}$ is a multiple of $x_{n-1}$ since they are both vectors in $\calH_{n-1}$.
\end{remarks}

\begin{proof} We start taking the inner product of $A_nx_{n-2}$ with \eqref{6.6} which given that $Ax_{n-2}\perp y_{n-1}$ and $\norm{A_nx_{n-2}}=1$ implies that
\begin{equation}\label{6.10}
  \jap{Ax_{n-2},x_{n-1}} = \rho_{n-2}
\end{equation}
Since $\bdone-P_{n-1} = \jap{x_{n-1},\cdot}x_{n-1}$, where $P_{n-1}$ be the projection on $\calH_{n-1}$, this implies that
\begin{equation}\label{6.11}
  (\bdone-P_{n-1})Ax_{n-2} = \rho_{n-2}x_{n-1}
\end{equation}

On the other hand, let $w \equiv P_{n-1}Ax_{n-2}$ and let $u\in \calH_{n-2}$.  Since $A_n$ preserves inner products on $\calH_{n-1}$, we see that $\jap{A_nu,A_nx_{n-2}}=0$.  Since $A_nu \in \calH_{n-1}$, we see that $P_{n-1}A_nu=A_nu$.  Thus
\begin{align*}
  \jap{A_nu,w} & =\jap{A_nu,P_{n-1}A_nx_{n-2}} = \jap{P_{n-1}A_nu,A_nx_{n-2}}\\
  & = \jap{A_nu,A_nx_{n-2}}=0
\end{align*}
Since $w \in \calH_{n-1}$, we conclude that $w$ is a multiple of $y_{n-2}$.  Therefore, for some $b$, we have that
\begin{equation}\label{6.12}
  Ax_{n-2} = (\bdone-P_{n-1})Ax_{n-2}+P_{n-1}Ax_{n-2}= \rho_{n-2}x_{n-1}+by_{n-2}
\end{equation}
Since $y_{n-2}\perp x_{n-1}$, we have that $b^2+\rho_{n-1}^2=1$, so we can pick the phase of $y_{n-2}$ so that \eqref{6.8} holds.

With this choice, we need to prove \eqref{6.9}.  Actually, if $a_{n-2}=0$ the above argument doesn't fix the phase of $y_{n-2}$ so in that case we will show that \eqref{6.9} can be used to determine the phase.  Taking the inner product of \eqref{6.6} with $y_{n-1}$ and using that $y_{n-1}\perp Ax_{n-2}$ we get that
\begin{equation}\label{6.13}
  \jap{y_{n-1},x_{n-1}} = -a_{n-2}
\end{equation}
Taking the inner product of \eqref{6.8} with $y_{n-1}$ and again using $y_{n-1}\perp Ax_{n-2}$ we get that
\begin{equation}\label{6.14}
  0=\rho_{n-2}\jap{y_{n-1},x_{n-1}}+a_{n-2}\jap{y_{n-1},y_{n-2}}
\end{equation}
Multiplying \eqref{6.13} by $\rho_{n-2}$ we conclude that
\begin{equation}\label{6.15}
  a_{n-2}\jap{y_{n-1},y_{n-2}}=a_{n-2}\rho_{n-2}
\end{equation}
which implies \eqref{6.9} if $a_{n-1}\ne 0$.

If $a_{n-2}=0$, i.e. $\rho_{n-2}=1$, then by \eqref{6.13}, $y_{n-1}\perp x_{n-1}$ and, thus, $y_{n-1} \in \calH_{n-1}$.  Since it is orthogonal to $A_n[\calH_{n-1}]$, it is also orthogonal to $A_n[\calH_{n-2}]$ and thus we can pick $y_{n-1}=y_n$ so that \eqref{6.9} holds.

\end{proof}

\begin{proof} [Second Proof of Theorem \ref{T2}] By iterating the above construction, having made a choice of phase for $x_{n-1}$ and $y_{n-1}$, we get $\calH_1 \subset\calH_2\subset\dots\subset\calH_n$ with $\dim(\calH_j)=j$ and $x_j,y_j\in\calH_{j+1},\,j=0,1,\dots,n-1$.  Moreover $x_j\perp\calH_j$ so $\{x_j\}_{j=0}^{n-1}$ is an orthonormal basis.

Changing the choice of phase of $y_{n-1}$ by replacing it by $e^{i\theta}y_{n-1}$ replaces each $y_j$ by $e^{i\theta}y_j$ (same $\theta$) since we rely on $\jap{y_j,y_{j-1}}>0$.  Moreover, $x_0,y_0\in\calH_1$, a one dimensional space, so we can choose $e^{i\theta}$ so that $y_0=x_0$.  We make that choice once and for all, which changes the phases of the $a_j$ and then we let $\alpha_j=\bar{a}_j$.

By construction, we have for $j=0,\dots,n-2$
\begin{align}
   Ax_j &=\rho_jx_{j+1}+\bar{\alpha}_j y_j \label{6.18} \\
   x_{j+1} &= \rho_j Ax_j-\bar{\alpha}_j y_{j+1} \lb{6.19}
 \end{align}
and by \eqref{6.2b} that
\begin{equation}\label{6.20}
  Ax_{n-1} = \bar{\alpha}_{n-1}y_{n-1}
\end{equation}

Multiplying \eqref{6.18} by $\rho_j$ and substituting in \eqref{6.19} gives
\begin{equation}\label{6.21}
  x_{j+1} = \rho_j^2 x_{j+1} + \rho_j\bar{\alpha}_j y_j - \bar{\alpha}_j y_{j+1}
\end{equation}
Since $1-\rho_j^2=\alpha_j\bar{\alpha}_j$, if $\alpha_j\ne 0$, we can divide it out and get
\begin{equation}\label{6.22}
  y_{j+1} = \rho_j y_j - \alpha_j x_{j+1}
\end{equation}
If $\alpha_j=0$, we saw that $y_{j+1}=y_j$ so \eqref{6.22} still holds.

We thus have \eqref{6.18}/\eqref{6.20}/\eqref{6.22} which are the same as \eqref{2.24}/\eqref{2.25} which can be used with $x_0=y_0$ to show that, in $\{x_j\}_{j=0}^{n-1}$ basis, $A$ is given by a GGT matrix proving Theorem \ref{T2}.
\end{proof}

\begin{proof} [Proof of Theorem \ref{T3}] The result is essentially a restatement of Wendroff Theorem for OPUC.  Given the eigenvalues $\{z_j\}_{j=1}^n$, take $\Phi_n(z)=\prod_{j=1}^{n}(z-z_j)$.  By Wendroff's Theorem, this is an OPUC for measure and if $A$ is the corresponding compressed multiplication operator, then its eigenvalues are the zeros of $\Phi_n$, so the required set. In fact, one proof of Wendroff's theorem \cite[Theroem 1.7.5]{OPUC1} shows that one can take the measure to be $d\theta/|\varphi_n|^2$, the so-called Bernstein--Szeg\H{o} measure associated to $\varphi_n$.

By Theorem \ref{T2.9}, two compressed multiplication operators with the same eigenvalues have the same $\Phi_n$ and are the exact same operator acting on $\calP_{n-1}$.  Thus, they are unitarily equivalent.
\end{proof}

%%%%%%%%%%%%%%%%%%%%%%%%%%%%%%%%%%%%%%%%%%%%%%%%%%%%%%%%%%%%%%
\section{The Numerical Range} \lb{s7}
%%%%%%%%%%%%%%%%%%%%%%%%%%%%%%%%%%%%%%%%%%%%%%%%%%%%%%%%%%%%%%

In this section, we prove Theorems \ref{T4} and \ref{T5}.  Since we've seen that $S_n$ agrees with compressed multiplication operators, we suppose $A$ acts on $\calP_{n-1}$ as a compressed multiplcation operators with Verblunsky coefficients $\{\alpha_j\}_{j=0}^{n-1}$.  We will realize the unitary operators $\{U_\lambda\}_{\lambda\in\partial\bbD}$ associated to the degree $n+1$ POPUC as operators on $\calP_{n}$.  We have $\{\varphi_j\}_{j=0}^{n}$ as an orthonormal basis for $\calP_{n}$ in $L^2(\partial\bbD,d\mu_\lambda)$ for the measures, $d\mu_\lambda$, associated to each of $U_\lambda$.  For $\psi\in\calP_n$, we have that
\begin{equation}\label{7.1}
  \psi\in\calP_{n-1} \iff \jap{\varphi_n,\psi}=0
\end{equation}

Szeg\H{o} recursion says that
\begin{equation}\label{7.2}
  A\varphi_j=\left\{
               \begin{array}{ll}
                 \rho_j\varphi_{j+1}+\bar{\alpha}_j\varphi_j^*, & \hbox{ } j=0,\dots, n-2 \\
                 \bar{\alpha}_j\varphi_j^*, & \hbox{ } j=n-1
               \end{array}
             \right.
\end{equation}
while
\begin{equation}\label{7.3}
  U_\lambda\varphi_j=\left\{
                       \begin{array}{ll}
                         \rho_j\varphi_{j+1}+\bar{\alpha}_j\varphi_j^*, & \hbox{ }j=0,\dots,n-1 \\
                         \bar{\lambda}_j\varphi_j^*, & \hbox{ }j=n
                       \end{array}
                     \right.
\end{equation}

$U_\lambda$ has eigenvalues at those $\{w_j\}_{j=1}^{n+1}$, labelled so that $0\le \arg w_j<\arg w_{j+1} < 2\pi, j=1,\dots,n$, with $\Phi_{n+1}(w_j)=0$, where $\Phi_{n+1}(z)=z\Phi_n(z)-\bar{\lambda}\Phi_n^*(z)$.  The eigenvectors are
\begin{equation}\label{7.4}
  \eta_j(z) = N_j^{-1} \Phi_{n+1}(z)/(z-w_j) = N_j^{-1} \prod_{k\ne j}(z-w_k)
\end{equation}
where $N_j>0$ is a normalization factor. Of course $\Phi_{n+1},w_j,\eta_j,N_j$ are all $\lambda$ dependent but we surpress this dependence unless we need to be explicit.

We define
\begin{equation}\label{7.5}
  m_j(\lambda)=|\jap{\varphi_n, \eta_j}|^2
\end{equation}
the Fourier coefficients of $\varphi_n$ in the orthonormal basis $\{\eta_j\}_{j=1}^{n+1}$ so
\begin{equation}\label{7.6}
  m_j(\lambda)>0,\qquad \sum_{j=1}^{n+1}m_j(\lambda)=1
\end{equation}
The spectral measure for $\varphi_n$ and $U_\lambda$ is
\begin{equation}\label{7.7}
  d\nu^{\lambda}(z) = \sum_{j=1}^{n+1} m_j \delta_{w_j} = |\varphi_n(z)|^2d\mu^\lambda(z)
\end{equation}
where
\begin{equation}\label{7.8}
  d\mu^\lambda(z)=\sum_{j=1}^{n+1} q_j \delta_{w_j};\qquad q_j = 1/\sum_{k=0}^{n+1} |\varphi_k(w_j)|^2
\end{equation}
(the Christoffel numbers, $q_j^{-1}$, here sum to $n+1$ while in Section \ref{s4}, they only summed to $n$ because there we were discussing $n\times n$ unitaries whereas here our unitary operators are $(n+1)\times(n+1)$).

We won't need them but we note there are explicit formulae for $N_j$ and $m_j$, viz
\begin{equation}\label{7.9}
  N_j=q_j^{1/2}\prod_{k\ne j}|w_j-w_k|;\qquad m_j=q_j|\varphi_n(w_j)|^2
\end{equation}
We will need that these and the $w_j$ are real analytic in $\lambda$ although one can also get that from eigenvalue perturbation theory \cite{OT}.

\begin{lemma} \lb{L7.1} Up to phase factor, the span of $\eta_j$ and $\eta_k$, $j\ne k$ has a unique unit vector, $\psi$, in $\calP_{n-1}$ and it is given by
\begin{equation}\label{7.10}
  \psi=[\jap{\varphi_n,\eta_j}\eta_k-\jap{\varphi_n,\eta_k}\eta_{j}]/\sqrt{m_j+m_k}
\end{equation}
One has that
\begin{equation}\label{7.11}
  \zeta_{jk}:=\jap{\psi,U_\lambda\psi}= \jap{\psi,A\psi}=\frac{m_jz_k+m_kz_j}{m_j+m_k}
\end{equation}
\end{lemma}

\begin{proof} \eqref{7.10} is a consequence of \eqref{7.1}.  \eqref{7.11} for $U_\lambda$ follows from the Fourier expansion for $U_\lambda$ and that this equals $\jap{\psi,A\psi}$ follows from $A=P_nU_\lambda P_n$.
\end{proof}

\begin{theorem} [=Theorem \ref{T4}] \lb{T7.2}  For any $j, k \in  \{1,\dots,n+1\},\, j\ne k$, the point, $\zeta$, on the line between $w_j$ and $w_k$ with
\begin{equation}\label{7.12}
  |\zeta-w_j|/|\zeta-w_k|= m_j/m_k
\end{equation}
lies in $N(A)$.  For $k=j+1$ (or $j=n+1, k=1$), this is the only point on the line which lies in $N(A)$ and that line is tangent to $\partial N(A)$ at the point $\zeta$.
\end{theorem}

\begin{remark} The polygon with vertices $\{w_j\}_{j=1}^{n+1}$ is $\partial N(U_\lambda)$ and circumscribes $N(A),$ so it is a kind of Poncelet polygon.
\end{remark}

\begin{proof} Given the formula \eqref{7.11} for $\zeta$, \eqref{7.12} is a direct calculation.  Since $\psi\in\calP_{n-1}$, we have that $\zeta\in N(A)$.

Since $U_\lambda$ is normal, $N(U_\lambda)$ is the convex hull of the eigenvalues $\{w_j\}_{j=1}^{n+1}$ and $\jap{\psi,U_\lambda\psi} = \sum_{j=1}^{n+1} w_j |\jap{\psi,\eta_j}|^2$ is in the segment from $w_j$ to $w_{j+1}$ if and only if $\psi$ is the the space spanned by $\eta_j$ and $\eta_{j+1}$.  Thus the only point on that line in $N(A)$ comes from $\psi$ given by \eqref{7.10} with $\jap{\psi,A\psi}$ given by \eqref{7.11}.  Since $N(A) \subset N(U_\lambda)$, $N(A)$ lies on one side of the line segments and therefore the line segment is tangent.
\end{proof}

\begin{theorem} [=Theorem \ref{T5}] \lb{T7.3} $\partial N(A)$ is a real analytic curve.  Each $\zeta\in\partial N(A)$ is a tangent point of an edge of some $N(U_\lambda)$.  Moreover
\begin{equation}\label{7.13}
  N(A) = \cap_{\lambda\in\partial\bbD} N(U_\lambda)
\end{equation}
\end{theorem}

\begin{proof} As noted above, the $w_j$ and $m_j$ are real analytic functions of $\bar{\lambda}$, so $\zeta$ is real analytic.  Since $\partial N(A)$ is a convex curve, there is, in the sense of a line with the curve on one side of the line, a tangent at each point (although, a priori, it might not be unique).  This tangent must meet $\partial\bbD$ at a point which is a zero of some $\Phi_{n+1}(z;\lambda)$ and so the tangent is an edge of that $N(U_\lambda)$.  That means that the function $\zeta$ fills out the entire set $\partial N(A)$.  The analyticity then implies uniqueness of tangent and so proves strict convexity.  Thus, we need only prove \eqref{7.13}.

Since $A$ is a restriction of each $U_\lambda$, we have that $N(A)\subset N(U_\lambda)$, so the $\subseteq$ of \eqref{7.13} is immediate.

Pick $\xi\in N(A)^{int}$, the interior of $N(A)$ (which is non--empty by the strict convexity).  Let $u\notin N(A)$.  The line segment from $\xi$ to $u$ meets $\partial N(A)$ in a unique point, $\zeta$.  $\zeta$ lies on the edge of some $\partial N(U_\lambda)$.  Since this edge is tangent to $\partial N(A)$, it must be distnct from the line from $\xi$ to $u$ which implies that $u \notin N(U_\lambda)$.  Thus $N(A)^c \subset \cup N(U_\lambda)^c$ proving the $\supseteq$ half of \eqref{7.13}.
\end{proof}

%%%%%%%%%%%%%%%%%%%%%%%%%%%%%%%%%%%%%%%%%%%%%%%%%%%%%%%%%%%%%%
\section{The Schur functions Associated to OPUC} \lb{s8}
%%%%%%%%%%%%%%%%%%%%%%%%%%%%%%%%%%%%%%%%%%%%%%%%%%%%%%%%%%%%%%

In this section and the next, we will prove Theorems \ref{T6} and \ref{T7} and the following related result

\begin{theorem} \lb{T8.1} Fix Verblunsky coefficients $\{\alpha_j\}_{j=0}^{n-1}$ and corresponding OPUC $\{\Phi_j\}_{j=0}^{n}$.  Let
\begin{equation}\label{8.1}
  B_n(z) = \frac{\Phi_n(z)}{\Phi_n^*(z)}
\end{equation}
Then $B_n$ is a Schur function and for any $\lambda\in\partial\bbD$, the Schur iterates of $\lambda B_n$ are $\lambda B_{n-1},\lambda B_{n-2},\dots,\lambda B_0 = \lambda$ and the Schur parameters are
\begin{equation}\label{8.2}
  \gamma_j(\lambda B_n) = \left\{
                            \begin{array}{ll}
                              -\lambda \bar{\alpha}_{n-1-j}, & \hbox{ } j=0,\dots,n-1 \\
                              \lambda, & \hbox{ } j=n
                            \end{array}
                          \right.
\end{equation}
The Carath\'{e}odory function for the associated measure is
\begin{equation}\label{8.3}
  F_n(z)=-\frac{\Phi_{n+1}(z;-\lambda)}{\Phi_{n+1}(z;\lambda)}
\end{equation}
The measure associated to $\lambda B_n$ is the measure $d\nu^\lambda$ of \eqref{7.7}.
\end{theorem}

In this section, we'll prove this theorem and use it to give the first proof of Theorem \ref{T7}.  In the next section, we'll use it to give one proof of Theorem \ref{T6}.  In that section we'll also give GGT matrix proofs of Theorems \ref{T6} and \ref{T7}. Formulae like \eqref{8.2} are implicit in the work of Khrushchev  \cite{Khr2001}.  We'll say more about his work and its relation to the first proof of Theorem \ref{T7} in the next Section.

\begin{proof} [Start of Proof of Theorem \ref{T8.1}] If $e^{i\theta} \in\partial\bbD$, the have that $|\Phi_n(e^{i\theta})|=|\Phi_n^*(e^{i\theta})|$.  Moreover, since $\Phi_n$ has all of its zeros in $\bbD$, $\Phi_n^*$ has none, so $B_n(z)$ is analytic in $\bbD$ with $|B_n(e^{i\theta})|=1$ for $e^{i\theta} \in\partial\bbD$.  By the maximum principle, $B_n$ is a Schur function; indeed, up to a phase factor, it is the Blaschke product of the zeros of $\Phi_n$.

By Szeg\H{o} recursion, \eqref{2.3} and the result of applying $\tau_{n+1}$ to it, we have that
\begin{align}
  \lambda B_n(z) &= \lambda\frac{z\Phi_{n-1}(z)-\bar{\alpha}_{n-1}\Phi_{n-1}^*(z)}{\Phi_{n-1}^*(z)-\alpha_{n-1}z\Phi_{n-1}(z)} \nonumber  \\
                 &= \frac{-\lambda\bar{\alpha}_{n-1}+z(\lambda B_{n-1}(z))}{1-\bar{\lambda}\alpha_{n-1}z(\lambda B_{n-1}(z))}  \label{8.4}
\end{align}
which precisely says that $\gamma_0(\lambda B_n)=-\lambda\bar{\alpha}_{n-1}$ and that $\lambda B_{n-1}$ is the first Schur iterate.  By the obvious repetition, we get \eqref{8.2} and the claimed full list of Blaschke iterates.

To get \eqref{8.3}, we note that
\begin{align}
  F_n(z) &= \frac{1+z\lambda B_n(z)}{1 - z\lambda B_n(z)} \nonumber \\
         &= \frac{z\Phi_n(z)+\bar{\lambda}\Phi_n^*(z)}{-(z\Phi_n(z)-\bar{\lambda}\Phi_n^*(z))} \label{8.5}
\end{align}
proving \eqref{8.3}.  This proves the entire theorem except for the identification of the measure to which we now turn.
\end{proof}

We have just shown that the Carath\'{e}odory function of $d\nu$, the measure with Schur function $\lambda B_n$ is
\begin{equation}\label{8.6}
  F(z) = \frac{1+\lambda zB_n(z)}{1-\lambda zB_n(z)}= -\frac{z\varphi_n(z) + \bar{\lambda}\varphi_n^*(z)}{z\varphi_n(z)-\bar{\lambda}\varphi_n^*(z)}
\end{equation}

On the other hand, the Carath\'{e}odory function of  the measure defined by $\Phi_{n+1}(z;\lambda)$ is, by \cite[(3.2.4)]{OPUC1},
\begin{equation}\label{8.7}
  G(z) = \frac{\Psi^*_{n+1}(z)}{\Phi^*_{n+1}(z)}=-\frac{\Psi_{n+1}(z)}{\Phi_{n+1}(z)}
\end{equation}
where we used \eqref{2.21}.  Here $\Psi_{n+1}(z;\lambda) \equiv z\Psi_n(z)+\bar{\lambda}\Psi^*_{n}(z)$ with $\Psi_n$ the second kind OPUC; see Theorem \ref{T2.6}.

\begin{lemma} \lb{L8.2} Fix $\lambda$ and $z_0$ in $\partial\bbD$.  Suppose that $\Phi_{n+1}(z_0;\lambda)=0$.  Then
\begin{equation}\label{8.8}
  \frac{1}{\Psi_{n+1}(z_0)}=\frac{\overline{\varphi_n(z_0)}}{2z_0\prod_{j=0}^{n-1}\rho_j}
\end{equation}
\end{lemma}

\begin{proof} By \eqref{2.13}, we have that
\begin{equation}\label{8.9}
  \Psi^*_n(z_0)\Phi_n(z_0)+\Phi^*_n(z_0)\Psi_n(z_0)= 2z_0^n \prod_{j=0}^{n-1}\rho_j^2
\end{equation}
By hypothesis, we have that $\Phi_n(z_0)=\overline{z_0\lambda}\Phi^*_n(z_0)$, so \eqref{8.9} becomes
\begin{equation}\label{8.10}
  \overline{z_0}\Phi_n(z_0)^*\Psi_{n+1}(z_0) = 2z_0^n\prod_{j=0}^{n-1}\rho_j^2
\end{equation}
Since $\Phi^*_n(z_0)=z_0^n\overline{\varphi_n(z_0)}\prod_{j=0}^{n-1}\rho_j$ we get \eqref{8.8}.
\end{proof}

\begin{proof} [End of the Proof of Theorem \ref{T8.1}] Let $z_0$ be a zero of the POPUC $\Phi_{n+1}(z;\lambda)$.  It suffices to show that for any such $z_0$, one has that
\begin{equation}\label{8.11}
  \lim_{\epsilon\downarrow 0} F((1-\epsilon)z_0)/G((1-\epsilon)z_0) = |\varphi_n(z_0)|^2
\end{equation}
We note that by \eqref{8.6} and \eqref{8.7}, we have that
\begin{equation}\label{8.12}
  \frac{F(z)}{G(z)}= \frac{z\Phi_n(z)+\bar{\lambda} \Phi_n^*(z)}{\Psi_{n+1}(z)}
\end{equation}

The right side of \eqref{8.12} has a limit at $z_0$ which is $2z_0\Phi_n(z_0)/\Psi_{n+1}(z_0)$.  By \eqref{8.8}, this is $|\varphi_n(z_0)|^2$.
\end{proof}

%%%%%%%%%%%%%%%%%%%%%%%%%%%%%%%%%%%%%%%%%%%%%%%%%%%%%%%%%%%%%%
\section{M-functions of POPUC} \lb{s9}
%%%%%%%%%%%%%%%%%%%%%%%%%%%%%%%%%%%%%%%%%%%%%%%%%%%%%%%%%%%%%%

In this section, we'll begin with a proof of Theorem \ref{T6} following up on the last section and then provide a totally different approach to proving Theorems \ref{T6} and \ref{T7} using GGT matrices.  Finally, we'll discuss the relation of these theorems to earlier work on OPUC.

We'll call a function like
\begin{equation}\label{9.1}
  \int \frac{1}{z-e^{i\theta}} \, d\nu(\theta)
\end{equation}
which appears on the left side of \eqref{3.2} an $M$-function in analogy with the Weyl $m$-function of OPRL and the theory of second order ODEs (although those functions have $(x-z)^{-1}$, not $(z-x)^{-1}$ where $x$ is the variable of integration).

\begin{proof} [First Proof of Theorem \ref{T6}] Since $\tfrac{w+z}{w-z} = 1-2z\tfrac{1}{z-w}$, we have that
\begin{equation}\label{9.2}
  F(z)=1-2zM(z) \Rightarrow M(z) = \frac{1-F(z)}{2z}
\end{equation}
where $F$, the Carath\'{e}odory of $d\nu$, is given by \eqref{8.6}.  \eqref{9.2} and \eqref{8.5} imply \eqref{3.2}.

This provides a relation between a POPUC and the measure, $d\nu$, associated to it.  Given a Blaschke product or the POPUC associated to it (the one with the same zeros), one sees from this that the combination on the right of \eqref{1.3} or \eqref{3.2} has the form on the left of these equations.  Conversely, given a set of $n + 1$ points on $\partial \mathbb D$ and probability weights, we can form the associated measure $\nu$ and its $M$-function $M$ as in (9.1). By (2.9) and (9.2), $M$ can be expressed in terms of a Schur function. It remains to use Theorem \ref{T8.1} to show that $M$ has the form of the right hand side of \eqref{1.3}, or equivalently, of \eqref{3.2}.
\end{proof}

The following direct proof of \eqref{3.2} does not rely on identifying a Schur function.  We regard it as the simplest proof of \eqref{1.3}.  After we'd found this proof, we found almost the identical argument in \cite[Section 2.2]{GolKud}.

\begin{proof} [Second Proof of Theorem \ref{T6}] Let $U_\lambda$ be given by Theorem \ref{T2.10}.  Then by the spectral theorem and the definition of $d\nu$, we have that
\begin{equation}\label{9.3}
  \int \frac{1}{z-e^{i\theta}} \, d\nu(\theta) = \jap{\varphi_n, (z-U_\lambda)^{-1}\varphi_n}
\end{equation}

In the GGT representation, $\varphi_n$ is the vector $(0,0,\dots,1)^t$, so Cramer's rule says that
\begin{equation}\label{9.4}
  \jap{\varphi_n,(z-U_\lambda)^{-1}\varphi_n} = \frac{\det(z-\calG_n(\{\alpha\}_{j=0}^{n-1}))}{\det(z-\calG_{n+1}(\{\alpha_j\}_{j=0}^{n-1}\cup \{\lambda\}))}
\end{equation}
since the result of dropping the last row and column of GGT matrix is a GGT matrix of one degree less.  By \eqref{2.19}, the denominator of \eqref{9.4} is $\Phi_{n+1}(z;\lambda)$ and by \eqref{2.16}, the numerator is $\Phi_n(z)$.  These facts imply \eqref{3.2}.  Once one has this, the rest of the proof is the same as the last paragraph of the first proof.
\end{proof}

\begin{proof} [Second Proof of Theorem \ref{T7}] Except for replacing $n$ by $n+1$ and $\alpha_{n-1}$ by $\lambda$, this is just Theorem \ref{T4.4}.
\end{proof}

Finally, we turn to the connection of Theorems \ref{T6}, \ref{T7} and \ref{T8.1} to earlier work of Khrushchev \cite{Khr2001} and Simon \cite{OPUC1}.  When $\lambda=1$, \eqref{8.2}, proven as we do using Szeg\H{o} recursion, is in Khrushchev's paper.  He uses it in part to prove what \cite[Theorem 9.2.4]{OPUC2} calls Khrushchev's formula that if $d\eta$ is a non--trivial probability measure, then the Schur function of $|\varphi_n(e^{i\theta})|^2\,d\eta(\theta)$ is $f_n(z)B_n(z)$ where $f_n$ is the $n$th Schur iterate of the Schur function of $d\eta$, i.e. the Schur function with Schur parameters $\gamma_j(f_n)=\alpha_{n+j}(d\eta), j=0,1,\dots$, and $B_n$ is given by \eqref{8.1}.

If now, we take $\alpha_n\to\lambda\in\partial\bbD$, the measure $d\eta$ converges to the measure we called $d\mu^\lambda$, so $|\varphi_n(e^{i\theta})|^2 d\eta$ converges to the measure that we called $d\nu^\lambda$ and the Khrushchev formula in the limit says that the Schur function of that measure is $\lambda B_n$ which is the final assetion in Theorem \ref{T8.1} which we saw above is the essence of the first proof of Theorem \ref{T6}.

In \cite[Section 4.4]{OPUC1}, Simon, using rank two perturbation theory computed matrix elements  of $(\calC-z)^{-1}$, where $\calC$ is the CMV matrix, which can be used to also find matrix elements of $(\calG-z)^{-1}$ which provides another proof of Theorem \ref{T6}.

%%%%%%%%%%%%%%%%%%%%%%%%%%%%%%%%%%%%%%%%%%%%%%%%%%%%%%%%%%%%%%
\section{Wendroff Type Theorems} \lb{s10}
%%%%%%%%%%%%%%%%%%%%%%%%%%%%%%%%%%%%%%%%%%%%%%%%%%%%%%%%%%%%%%

Simon \cite{SimonPOPUC} has three theorems about zero interlacing involving POPUC which he proves using rank one perturbation theory (recently Castillo--Petronilho \cite{Cast3} have extended some of these results and recast them.)

(1) (proven earlier by Cantero et al \cite{CMV}; see also \cite{GolKud}). If $\lambda,\mu\in\partial\bbD$ are different, then the zeros of $\Phi_{n+1}(z;\lambda)$ and $\Phi_{n+1}(z;\mu)$ strictly interlace.

(2) (proven independently by Wong \cite{Wong}) If $\lambda\in\partial\bbD$, then the zeros of the first and second kind POPUC, $\Phi_{n+1}(z;\lambda)$ and $\Psi_{n+1}(z;\lambda)$, strictly interlace (we remark that here we define $\Psi_{n+1}(z;\lambda)$ to have all the opposite sign Verblunsky coefficients to $\Phi_{n+1}(z;\lambda)$, including changing $\alpha_n=\lambda$ to $\alpha_n=-\lambda$, which \cite{SimonPOPUC,Wong} call $\Psi_{n+1}(z;-\lambda)$).

(3) (special case appeared earlier in Golinskii \cite{Golin} and a refined version in \cite{Cast3}). If $\lambda,\mu\in\partial\bbD$, perhaps equal, the zeros of the POPUC $\Phi_{n+1}(z;\lambda)$ and $\Phi_{n}(z;\mu)$ are either all distinct in which case $\Phi_{n}(z;\mu)$ has exactly one zero in exactly $n$ of the $n+1$ intervals obtained by removing the zeros of $\Phi_{n+1}(z;\lambda)$ from $\partial\bbD$ or else they have exactly one zero in common in which case $\Phi_{n}(z;\mu)$ has no zero in the two intervals obtained by removing the zeros of $\Phi_{n+1}(z;\lambda)$ from $\partial\bbD$ closest to the common zero and exactly one zero in the other $n-1$ intervals.

Wendroff type theorems mean suitable unique converses and, in this regard, parameter counting is important.  The various parameters lie in real manifolds and if there is to be existence and uniqueness, the two manifolds must have the same real dimension.

In this regard (1) is fine.  The set of possible zeros has real dimension $2n+2$ and they are determined by $n$ complex Verblunsky coefficients $\{\alpha_j\}_{j=0}^{n-1}, \lambda$ and $\mu$, also $2n+2$ real parameters.  Indeed Theorem \ref{T8} is such a Wendroff theorem which we noted is essentially identical to a theorem of Gau-Wu \cite{GW2004A} which we called Theorem \ref{TC}. It was later proven in a simpler way by Daepp et al \cite{Gorkin2010}.  The Daepp et al proof is essentially a POPUC analog of a standard proof of Wendroff for OPRL. Our proof (below) is new, albeit close to the earlier OPUC proof of Golinskii--Kudryavtsev \cite{GolKud}, and we feel illuminating.

As stated, (2) doesn't have the right parameter counting for a Wendroff theorem.  The zeros are again $2n+2$ parameters but $n$ complex Verblunsky coefficients $\{\alpha_j\}_{j=0}^{n-1}$ and $\lambda$ are only $2n+1$ real parameters.  That's because there is a restriction on the zeros, namely \eqref{3.4}.  With this restriction, the set of zeros is only $2n+1$ parameters so parameter counting is fine and there is a Wendroff theorem, Theorem \ref{T9}, new here.

The parameter counting is also wrong for there to be a strict converse in case (3).  There are $2n+1$ real parameters for the zeros but $n$ complex Verblunsky coefficients $\{\alpha_j\}_{j=0}^{n-1}, \lambda$ and $\mu$, are $2n+2$ real parameters.  Indeed some simple examples when $n=1$ show that, in general, there exists a one parameter family of choices for $(\alpha_0,\alpha_1,\lambda,\mu)$ leading to zeros $(w_1,w_2,y_1)$.  Nevertheless, Castillo et al. \cite{Cast2} do have a Wendroff type theorem in the context of (3)!  They consider sequences of monic polynomials $\{\Xi_j\}_{j=0}^\infty$ obeying a three term recurrence relation
\begin{equation}\label{10.1}
  \Xi_{n+1}(z) = (z+\beta_n)\Xi_{n}(z)- \gamma_n\Xi_{n-1}(z)
\end{equation}
with $n=1,2,\dots$ with $\Xi_0=1$ and $\Xi_1(z) = z+\beta_0$ for suitable $\beta_j\in\partial\bbD, j \ge 0$ and $\gamma_j \in\bbC\setminus\{0\}, j\ge 1$. \cite{Cast2} show that, under suitable hypotheses, there exist Verblunsky coefficients $\{\alpha_j\}_{j=0}^\infty\in\bbD^\infty$ and $\{\lambda_j\}_{j=0}^\infty\in\partial\bbD^\infty$ so that $\Xi_j$ are the associated POPUC but, when they exist, $\{\alpha_j\}_{j=0}^\infty\in\bbD^\infty$ and $\{\lambda_j\}_{j=0}^\infty\in\partial\bbD^\infty$ are not unique but depend on an arbitrary choice of $\lambda_0\in\partial\bbD$.  They proved that given sets of zeros as in case (3), there is a unique choice of $\{\beta_j\}_{j=0}^n$ and $\{\gamma_j\}_{j=0}^n$ and so $\{\Xi_j\}_{j=0}^{n-1}$ for which $\Xi_n$ and $\Xi_{n+1}$ have the required zeros.  Thus there is a kind of Wendroff theorem for the sequence of POPUC but the OPUC are not determined.  In contradistinction, Theorem \ref{T8} does determine all the lower order OPUC from the two POPUC which is why we call it Wendroff's Theorem for (P)OPUC.

We now turn to the proof Theorem \ref{T8}.  We begin with uniqueness.  Since
\begin{equation}\label{10.2}
  \Phi_{n+1}(z,\lambda) = z\Phi_n(z) - \bar{\lambda}\Phi_n^*(z); \qquad \Phi_{n+1}(z,\mu) = z\Phi_n(z) - \bar{\mu}\Phi_n^*(z)
\end{equation}
we have that
\begin{equation}\label{10.3}
  \Phi_n(z) = \frac{\lambda \Phi_{n+1}(z;\lambda)-\mu \Phi_{n+1}(z;\mu)}{(\lambda-\mu)z}
\end{equation}
If $z_1, z_2,\dots, z_{n+1}, w_1, w_2,\dots, w_{n+1}$ are all in $\partial\bbD$ and are the zeros of $\Phi_{n+1}(z,\lambda)$ and $\Phi_{n+1}(z,\mu)$, then
\begin{equation}\label{10.4}
   \Phi_{n+1}(z;\lambda) = \prod_{j=1}^{n+1}(z-z_j);\qquad \Phi_{n+1}(z;\mu) = \prod_{j=1}^{n+1}(z-w_j)
\end{equation}
so by \eqref{10.3}, the $z's$ and $w's$ determine $\Phi_n(z)$ and so, by Wendroff's Theorem for OPUC (Theorem \ref{T2.3}), they determine $\{\alpha_j\}_{j=0}^{n-1}$.  By \eqref{2.20}, the zeros also determine $\lambda$ and $\mu$.  Thus, we have the uniqueness claim.

For existence, we fix $2n+2$ points $z_1, z_2,\dots, z_{n+1}, w_1, w_2,\dots, w_{n+1}$ all in $\partial\bbD$.  For now, they need not be different or in any order but we do need $\lambda$ and $\mu$ below to be unequal.

Define
\begin{equation}\label{10.5}
  Q_{n+1}(z) = \prod_{j=1}^{n+1}(z-z_j);\qquad R_{n+1}(z) = \prod_{j=1}^{n+1}(z-w_j)
\end{equation}
\begin{equation}\label{10.6}
  \lambda=-\prod_{j=1}^{n+1} (-\bar{z_j}) \qquad \mu=-\prod_{j=1}^{n+1} (-\bar{w_j})
\end{equation}
so that
\begin{equation}\label{10.7}
  Q_{n+1}(0) = - \bar{\lambda}; \qquad Q_{n+1}(0) = - \bar{\mu}
\end{equation}

We revert to $\tau^n$ notation for the $^*$ map since we'll be using it with different implicit $n$'s.  We claim that
\begin{equation}\label{10.8}
  \tau_{n+1}(Q_{n+1}) = -\lambda Q_{n+1}; \qquad \tau_{n+1}(R_{n+1}) = -\mu R_{n+1}
\end{equation}
One way of seeing this is to note both sides have the same zeros, are polynomials of degree $n+1$ and have the same leading coefficients.  Another way is to note that
\begin{equation*}
  z\left(\frac{1}{z} - \bar{z}_j\right) = 1-z\bar{z}_j = -\bar{z}_j(z-z_j)
\end{equation*}

Define now
\begin{equation}\label{10.9}
  P_n(z) = \frac{\lambda Q_{n+1}(z)-\mu R_{n+1}(z)}{(\lambda-\mu)z}
\end{equation}
Of course, this is motivated by \eqref{10.3}. We note that by \eqref{10.7} the numerator vanishes at $z=0$ so $P_n$ is a polynomial of degree $n$ and the $\lambda-\mu$ in the denominator makes it a monic polynomial.  We compute:
\begin{align}
  \tau_n(P_n) = \tau_{n+1}(zP_n) &= \frac{\bar{\lambda}(-\lambda) Q_{n+1} - \bar{\mu}(-\mu)R_{n+1}}{\bar{\lambda}-\bar{\mu}} \nonumber \\
                                 &= \mu\lambda \frac{Q_{n+1}-R_{n+1}}{\lambda-\mu}     \label{10.10}
\end{align}
where we used
\begin{equation*}
  \frac{1}{\bar{\lambda}-\bar{\mu}} = \frac{1}{\frac{1}{\lambda}-\frac{1}{\mu}} = -\frac{\lambda-\mu}{\mu\lambda}
\end{equation*}

This implies, by direct computation from \eqref{10.9} and \eqref{10.10}, that

\begin{lemma} \lb{L10.1} We have that
\begin{equation}\label{10.11}
  zP_n-\bar{\lambda}\tau_n(P_n) = Q_{n+1}; \qquad zP_n-\bar{\mu}\tau_n(P_n) = R_{n+1}
\end{equation}
\end{lemma}

\begin{lemma} \lb{L10.2} Let $z_0\in\partial\bbD$.  Then
\begin{equation}\label{10.12}
  P_n(z_0) = 0 \iff Q_{n+1}(z_0)=R_{n+1}(z_0)= 0
\end{equation}
\end{lemma}

\begin{proof} Since $z_0\in\partial\bbD$, one has that $P_n(z_0)=0 \iff \tau_n(P_n)(z_0) = 0$.  Given this, \eqref{10.12} is immediate from \eqref{10.9}, \eqref{10.10} and \eqref{10.11}.
\end{proof}

With this algebra under our belt

\begin{proof} [Proof of Theorem \ref{T8}] That the zero's interlace follows from the fact that $B_n(z)=\Phi_n(z)/\Phi_n^*$ has a strictly increasing argument.  Uniqueness is proven above.

Finally, given interlacing $z$'s and $w$'s, let $\ti{z}_j$ be the roots of $z^{n+1}-\bar{\lambda}$ and $\ti{w}_j$ be the roots of $z^{n+1}-\bar{\mu}$.  The $P$ for this set of roots is clearly $z^n$ which has all its roots in $\bbD$.  One can continuously deform these sets keeping the corresponding $\lambda$ and $\mu$ unchanged through strictly interlacing sets and so deform from $z^n$ to $P_n$.  The zeros of $P_n$ move continuously and, by Lemma \ref{L10.2} and strict interlacing, never go through $\partial\bbD$.  Thus $P_n$ has all its zeros inside $\bbD$ so, by Wendroff's theorem for OPUC (Theorem \ref{T2.3}), it is a $\Phi_n$.  By Lemma \ref{L10.1}, $Q_{n+1}$ and $R_{n+1}$ are the POPUC for the $\Phi_n$.
\end{proof}

Theorem \ref{T8} appears already in Golinskii--Kudryavtsev \cite[Theorem 3.2]{GolKud}.  Our proof of uniqueness is essentially the same as theirs.  Our proof of existence is related to theirs but instead of our perturbation argument above, they use a Theorem of Simon \cite[Theorem 11.5.6]{OPUC2}.  We believe that the role of interlacing is clearer in our new argument than in this earlier argument of Simon.

We now turn to the proof of Theorem \ref{T9}.  We will need \eqref{8.7} which expresses the Carath\'{e}odory function, $G(z)$, of the measure we called $d\mu^\lambda$ in terms of $\Phi_{n+1}(z;\lambda)$ and $\Psi_{n+1}(z;\lambda)$.  The Verblunsky coefficients of that measure are precisely what the second half of the Theorem claims are determined by the zeros obeying \eqref{3.4}.  The proof will depend on an analysis of what we'll call ``quasi-Carath\'{e}odory'' functions (because the $c$'s below may not be positive).  A degree $n+1$ quasi-Carath\'{e}odory function is one of the form, given $\{z_j\}_{j=1}^{n+1}\subset\partial\bbD$:
\begin{equation}\label{10.13}
  f(z) = \sum_{j=1}^{n+1}c_j \frac{z_j+z}{z_j-z}; \quad c_j\in\bbR\setminus\{0\};\quad \sum_{j=1}^{n+1}c_j = 1
\end{equation}

\begin{proposition} \lb{P10.3} Let $f$ be a quasi-Carath\'{e}odory function \eqref{10.13} whose poles are written in cyclic order with $z_{n+2}\equiv z_1$, $c_{n+2}\equiv c_1$.  Then

\mbox{(a)} $f(z)$ is purely imaginary on the unit circle with the $z_j$'s removed.

\mbox{(b)} $f(0)=1;\, f(\infty)=-1$

\mbox{(c)} For each $j=1,\dots,n+1$, if $c_j$ and $c_{j+1}$ have the same (resp, different) signs, then $f$ has an odd (resp. even) number of zeros (counting multiplicity)  in the arc between $z_j$ and $z_{j+1}$.
\end{proposition}

\begin{remark} If all $c_j$'s are of the same sign, by the pigeonhole principle, all the zeros of $f$  lie on $\partial\bbD$ and there is a single zero in each arc between $z_j$ and $z_{j+1}$. However, if we allow for a sign changes, it is easy to construct examples  where not all the zeros of $f$ lie on $\partial\bbD$ and examples where all of them do.
\end{remark}

\begin{proof} (a) and (b) are immediate. (c) By noting that $\mbox{Im}((1+z)/(1-z))= 2\mbox{Im}(z)/|1-z|^2$, one sees that if $c_j>0$, then $f\to i\infty$ as $z$ approaches $z_j$ from larger argument.  This implies that if $c_j$ and $c_{j+1}$ has opposite (resp. same) signs, there must be an even (resp. odd) number of sign changes of $\mbox{Im}(f)$ in the arc between $z_j$ and $z_{j+1}$.
\end{proof}

\begin{proposition} \lb{P10.4} Let $\{z_j\}_{j=1}^{n+1}$ and $\{w_j\}_{j=1}^{n+1}$ be two sets of points, all distinct, on the unit circle for which \eqref{3.4} holds.  Let
\begin{equation}\label{10.14}
  f(z) = \frac{\prod_{j=1}^{n+1}(1-\bar{w}_jz)}{\prod_{j=1}^{n+1}(1-\bar{z}_jz)}
\end{equation}
Then $f$ is a quasi-Carath\'{e}odory function.
\end{proposition}

\begin{proof} We begin by noting that if $|u|=1$, then
\begin{equation*}
  \overline{1-\bar{u}(1/\bar{z})}= (z-u)z^{-1} = -u(1-\bar{u}z)z^{-1}
\end{equation*}
which implies that
\begin{equation}\label{10.15}
  \overline{f\left(\frac{1}{\bar{z}}\right)} = \frac{ \prod_{j=1}^{n+1} w_j}{\prod_{j=1}^{n+1} z_j}f(z) = -f(z)
\end{equation}
by \eqref{3.4}.

Thus $f$ is imaginary on the circle away from the poles which implies, by a partial fraction expansion that
\begin{equation*}
  f(z) = c_0+\sum_{j=1}^{n+1}c_j \frac{z_j+z}{z_j-z}; \quad c_j\in\bbR\setminus\{0\}; c_0\in i\bbR
\end{equation*}
Since $f(0)=1$, we see that $c_0+\sum_{j=1}^{n+1}c_j=1$ which is real and so implies that $c_0=0$. Thus $f$ is a quasi-Carath\'{e}odory function.
\end{proof}

\begin{proof} [Proof of Theorem \ref{T9}] If $z_j$ and $w_j$ are the zero of a POPUC and its associated second kind POPUC, then  $\prod_{j=1}^{n+1} z_j = (-1)^n\bar{\lambda}$ while $\prod_{j=1}^{n+1} w_j = (-1)^n(-\bar{\lambda})$ proving \eqref{3.4}.  By \eqref{8.7} and Proposition \ref{P10.3}(c), there are an odd number of zeros in each of the $n+1$ intervals, so one per interval and the zeros interlace.  That proves the direct result.

For the converse, given interlacing zeros obeying \eqref{3.4}, by Proposition \ref{P10.4}, $f$ given by \eqref{10.14} is a quasi-Carath\'{e}odory function.  Since there are an odd number of zeros in each interval, the signs of the $c_j$'s are all the same.  Since they sum to $1$, they are all positive, so $f$ is the Carath\'{e}odory function of a point measure whose POPUC has zeros at the poles of $f$, so the $z_j$'s.  Moreover, the second kind POPUC has zeros at the zeros of $f$, so the $w$'s.
\end{proof}

%%%%%%%%%%%%%%%%%%%%%%%%%%%%%%%%%%%%%%%%%%%%%%%%%%%%%%%%%%%%%%
\section{Derivatives of POPUC} \lb{s11}
%%%%%%%%%%%%%%%%%%%%%%%%%%%%%%%%%%%%%%%%%%%%%%%%%%%%%%%%%%%%%%

There is a classical connection relating zeros of a polynomial, $P$, and zeros of its derivative, $P'$ to sums of the form \eqref{1.1} with all $m_j$'s equal because of the following basic relation.

\begin{proposition} \lb{P11.1} Let $P(z)$ be a complex polynomails of degree $k$ with distinct zeros $\{z_j\}_{j=1}^k$.  Then the zeros of $P'(z)$ are the same as the zeros of
\begin{equation}\label{11.1}
  M(z) \equiv \frac{1}{k}\sum_{j=1}^{k}\frac{1}{z-z_j}
\end{equation}
\end{proposition}

\begin{proof} Using $P(z)=a\prod_{j=1}^{k}(z-z_j)$ and taking logarithmic derivatives, one sees that
\begin{equation}\label{11.2}
  M(z) = \frac{P'(z)}{kP(z)}
\end{equation}
Since the zeros of $P$ are distinct, the zeros of $M$ are precisely the zeros of $P'(z)$.
\end{proof}

\begin{remark} Of course, \eqref{11.2} holds if some of the zeros of $P$ are not distinct, at least away from zeros of $P$.  But at multiple zeros of $P$, $M$ has a pole even though $P'$ vanishes.
\end{remark}

Given the relation of sums like \eqref{11.1} and Poncelet type polygons that we studied in section \ref{s7} and what Theorem \ref{T7.2} tells us when there are equal $m$'s, we are interested in polygons tangent at midpoints.  In this regard, there are two classical results in the case of triangles

\renewcommand\thetheorem{\Alph{theorem}}
\addtocounter{theorem}{4}

\begin{theorem} [Steiner \cite{Steiner}, 1829] \lb{TF} Given a triangle, $T$ in the complex plane with vertices $\{z_j\}_{j=1}^3$, there is a unique ellipse tangent to $T$ at the midpoints of its sides and the foci of the ellipse are
\begin{equation}\label{11.3}
  \frac{1}{3}(z_1+z_2+z_3) \pm \sqrt{\left(\frac{1}{3}(z_1+z_2+z_3)\right)^2-\frac{1}{3}(z_1z_2+z_1z_3+z_2z_3)}
\end{equation}
\end{theorem}

\begin{theorem} [Siebeck \cite{Sieb}, 1864]  \lb{TG} The foci of the Steiner ellipse are the zeros of $P'$ where $P$ is the cubic polynomial with zeros at $\{z_j\}_{j=1}^3$
\end{theorem}

\begin{remark} Of course, since $P(z)=(z-z_1)(z-z_2)(z-z_3)$, we have $P'(z)=az^2+bz+c;\, a=3, b=-2(z_1+z_2+z_3), c = (z_1z_2+z_1z_3+z_2z_3)$ so this result is immeduiate.  It was Siebeck who realized the special role of sums like \eqref{11.1} in studying the zeros of $P'$.
\end{remark}

\renewcommand\thetheorem{\arabic{theorem}}
\numberwithin{theorem}{section}
\addtocounter{theorem}{-6}

There are two papers in the series discussed in Section \ref{s1} that deal with extending these results to $k$ zeros on $\partial\bbD$ (for the case of three zeros, they always lie on a circle so there is no loss in supposing that $z_1,z_2,z_3\in\partial\bbD$), namely Gau-Wu \cite{GW1999} and Gorkin--Shubak \cite{Gorkin2011}.  Even though all the results later in this section are more or less in these papers, we decided to include this subject for several reasons: first, we believe the results are interesting and deserve to be better known; second, we want to emphasize how simple the basic proofs are; finally, and most importantly, while Gorkin--Shubak \cite{Gorkin2011} state the basic equation for determining when $n$ points in $\bbD$ are the zeros of a derivative of a polynomial, $P$, whose zeros lie in $\partial\bbD$ (Theorem \ref{T11.3}), they don't analyze them nor note the different nature in the case $n$ is odd vs. $n$ even (Theorem \ref{T11.4} below is new).

The following is a variant of the main result in Gau--Wu \cite{GW1999}.  Since the numerical range of an operator $A$ in $S_2$ is an ellipse (indeed, all operators on $\bbC^2$ have ellipses as their numerical range), the $n=2$ case of this theorem provides a proof of Theorems \ref{TF} and \ref{TG}.

\begin{theorem} \lb{T11.2} Let $\{z_j\}_{j=1}^{n+1}$ be $n+1$ distinct points in $\partial\bbD$.  Then up to unitary equivalence, there is a unique operator $A\in S_n$ so that $N(A)$ is tangent to the midpoints of the edges of the convex polygon with vertices $\{z_j\}_{j=1}^{n+1}$.  If $P(z) = \prod_{j=1}^{n+1}(z-z_j)$, then $(n+1)^{-1}P'(z)$ is the characteristic polynomial for $A$ and, in particular, the zeros of $P'$ are the eigenvalues of $A$.
\end{theorem}

\begin{proof} By Theorem \ref{T7.2}, there is a unique oprator $A\in S_n$ with $N(A)$ tangent at the midpoints of the polygon and its $M$--function has the form \eqref{11.1} with $k=n+1$.  By Theorem \ref{T6} and \eqref{11.2}, we have that
\begin{equation}\label{11.4}
  \frac{P'(z)}{P(z)} = \frac{\Phi_n(z)}{z\Phi_n(z)-\bar{\lambda}\Phi_n^*(z)}
\end{equation}
so $P'(z)=\Phi_n(z)$.  By Theorems \ref{T2.9} and \ref{T2}, $\Phi_n$ is the characteristic polynomial of $A$.
\end{proof}

Gorkin--Shubak \cite{Gorkin2011} study the question of when $\{a_j\}_{j=1}^n$ in $\bbD$ are the zeros of the derivative of a polynomial, $P$, all of whose zeros lie in $\partial\bbD$.  While they focus on the case $n=2$, they state the following (in a version that doesn't mention POPUC or OPUC!).

\begin{theorem} \lb{T11.3} Let $\{a_j\}_{j=1}^n$ be $n$ not necessarily distinct points in $\bbD$.  Then there is a polynomial, $P$, with all of its zeros on $\partial\bbD$ so that $P'$ has its zeros at $\{a_j\}_{j=1}^n$ if and only if there is a $\lambda\in\partial\bbD$ so that
\begin{equation}\label{11.5}
  \Phi_{n+1}'(z;\lambda) = (n+1)\Phi_n(z)
\end{equation}
where $\Phi_n(z)=\prod_{j=1}^{n}(z-a_j)$ and $\Phi_{n+1}(z;\lambda)$ is given by \eqref{2.18}.
\end{theorem}

\begin{proof} If \eqref{11.5} holds, then, clearly, $\{a_j\}_{j=1}^n$ is the set of critical points of a polynomial whose zeros all lie on $\partial\bbD$.

Converely, let $P$ be a monic polynomial whose zeros $\{z_j\}_{j=1}^{n+1}$ all lie in $\partial\bbD$ and so that $P'$ vanishes at $\{a_j\}_{j=1}^n$.  By Theorem \ref{T6}, we know that $M(z)$ given by \eqref{11.1} (with $k=n+1$) is the $M$--function associated to some $n+1$--point measure $d\mu$ on $\partial\bbD$ with OPUC obeying \eqref{3.2} with $d\nu=|\varphi_n|^2d\mu=(n+1)^{-1}\sum_{j=1}^{n+1}\delta_{z_j}$.  Moroever, for some $\lambda$, $z_j$ are the zeros of $\Phi_{n+1}(z;\lambda)$.  Thus $P(z)=\Phi_{n+1}(z;\lambda)$ and, by \eqref{3.2} and \eqref{11.2},
\begin{equation}\label{11.6}
  \frac{1}{n+1}\frac{P'(z)}{\Phi_{n+1}(z;\lambda)}=\frac{\Phi_n(z)}{\Phi_{n+1}(z;\lambda)}
\end{equation}
so \eqref{11.5} holds.
\end{proof}

Now, given $\{a_j\}_{j=1}^n$, a collection of points in $\bbD$, let $s_k=\sum a_{i_1}a_{i_2}\dots a_{i_k}$, with the sum over all $n\choose k$ sets of distinct $i_1 i_2 \dots i_k$ in $\{1,\dots n\}$, be the elementary symmetric functions, $k=1,\dots,n$ so that
\begin{equation}\label{11.7}
  \Phi_n(z) \equiv \prod_{j=1}^{n}(z-a_j) = z^n+\sum_{k=0}^{n-1}(-1)^{n-k} s_{n-k}z^k
\end{equation}

\begin{theorem} \lb{T11.4} \eqref{11.5} is equivalent to the equations
\begin{equation}\label{11.7}
  (n-j)s_{n-j}+(-1)^{n-1}\bar{\lambda}(j+1) \bar{s}_{j+1} = 0,\qquad j=0,\dots,n-1
\end{equation}
For $n=2k$, even, there are $k$ independent equations and in the $2n$ real dimensional manifold of $(a_1,\dots,a_n)\in\bbD^n$, the set obeying \eqref{11.5} has real dimension $n+1$.  For $n=2k+1$, odd, there are $k+1$ equations and if $s_{k+1} \ne 0$, we have that
\begin{equation}\label{11.8}
  \lambda = (-1)^{n-1} \bar{s}_k/s_k
\end{equation}
In the $2n$ real dimensional manifold of $(a_1,\dots,a_n)\in\bbD^n$, the set obeying \eqref{11.5} has real dimension $n+1$.
\end{theorem}

\begin{remarks} 1.  It may be that the set of allowed $a$'s has some singular points but except for a lower dimensional set, it is a real analytic manifold.

2.  The case $n=2$ is treated in detail in Gorkin--Shubak \cite{Gorkin2011}.  In that case, there is one equation where $\lambda$ can be used to adjust the phase so that we get
\begin{equation}\label{11.9}
  |a_1+a_2|=|a_1 a_2|
\end{equation}

3.  In terms of the operator, $A\in S_n$ associated to $\Phi_n$, the $a_j$ are eigenvalues and $s_k = \tr(\wedge^k(A))$ (see Simon \cite{TI} for discussion of $\wedge^k(A)$).  In particular, \eqref{11.7} for $j=0$ is
\begin{equation}\label{11.10}
  n\det(a)+(-1)^{n-1}\bar{\lambda}\tr(A^*)=0
\end{equation}

4.  The proof shows that the only possible singularities occur at points where two $a$'s are equal and, if $n=2k+1$ is odd, where $s_{k+1}=0$.  It is an interesting question whether these singularities actually occur.
\end{remarks}

\begin{proof} \eqref{11.5} is equivaelent to
\begin{equation}\label{11.11}
  n\Phi_n+\bar{\lambda}(\Phi_n^*)'-z\Phi_n = 0
\end{equation}
Write $\Phi_n(z)=z^n+\sum_{j=0}^{n-1} A_jz^j$.  Then $(\Phi_n^*)'(z)=\sum_{k=1}^{n}kA_{n-k}z^k$ and \eqref{11.11} becomes \eqref{11.8}.

It is well known that away from coincident points, the map $(a_1,\dots,a_n)\mapsto (s_1,\dots,s_n)$ has a non--singular inverse.  The equations \eqref{11.8} are the same for $j$ and for $n-1-j$ but are clearly independent linear equations for $j=1,2,\dots,\left[\frac{n}{2}+1\right]$.  The dimension counting results follow.
\end{proof}

%%%%%%%%%%%%%%%%%%%%%%%%%%%%%%%%%%%%%%%%%%%%%%%%%%%%%%%%%%%%%%
\section{Poncelet's Theorem} \lb{s12}
%%%%%%%%%%%%%%%%%%%%%%%%%%%%%%%%%%%%%%%%%%%%%%%%%%%%%%%%%%%%%%

That completes most of we want to say about using OPUC ideas to understand the topics discussed in Section \ref{s1}.  These final two sections make some brief comments about the relation of algebraic geometric ideas in the area where we have not succeeded in leveraging OPUC methods.

We begin with Poncelet's Theorem.  For the triangle case, the ideas of Section \ref{s7} are ideal.  Given two points $z_1, z_2 \in \bbD$, there is a unique monic polynomial, $\Phi_2(z)=(z-z_1)(z-z_2)$ with simple zeros at those two points (assuming $z_1\ne z_2$).  As with all $2 \times 2$ matrices, the numerical range of the associated compressed multiplication operator, $A$, is an ellipse with foci at the eigenvalues $z_1, z_2$.  We can form the unitaries associated to the various POPUC, $\Phi_3(z;\lambda)$ as $\lambda$ runs through $\partial\bbD$.  Their numerical ranges provide an infinity of Poncelet triangles with vertices on the unit circle and tangent to $N(A)$.

Moreover, if the ellipse is made larger (with the same foci), it is easy to see that starting at $w_0 \in \partial\bbD$ and forming three successive tangents to the large ellipse ending at successive points $w_1, w_2, w_3$ don't go as far around the circle as for the critical ellipse and so one gets less than a closed triangle.  Similarly, for a smaller ellipse, the $w_j$ are further along and so don't give a closed triangle.  Thus we only have a triangle for the critical ellipse which has an infinity of triangles.  This proves Poncelet's Theorem for triangles.

As we emphasized, the ideas of Section \ref{s7} give an analog of Poncelet's theorem from triangles to other polygons, but not the analog that Poncelet considered, since in that case, in general, the numerical range is not an ellipse.  That said, Mirman \cite{Mirman1998} and Gau--Wu \cite{GW2003} did use some of these ideas for a proof of the Poncelet ellipse theorem for higher degree polygons but they rely on a result of Kippenhahn \cite{Kipp} who considers the dual of the boundary, $\partial N(A)$, of the numerical range of an $n\times n$ matrix.  He considers the projective dual of this boundary, essentially, the set of tangents to the curve and proves that in projective coordinates, this dual is a degree $n$ real variety (or more properly the outer part of the variety).

The above argument for triangles extends to $n$-gons and shows that if we fix the foci of an ellipse in $\bbD$ and $w_0\in\partial\bbD$, then there is a single eccentricity where the $n$ touching tangents starting at $w_0$ lead to a closed convex $n$-gon.  This is perhaps best understood using the idea of a billiard trajectory as in \cite[Chapter 15]{PonceBk}.  Indeed, given $w_0\in\partial\bbD$ and an ellipse $E\subset\bbD$, there are two tangents to $E$ that pass through $w_0$.  Let $\ell$ be the tangent that intersects $\partial\bbD$ at $w_0$ and a point $w_1$ such that the path from $w_0$ to $w_1$ on $\partial\bbD$ in the counterclockwise direction is along the arc that is separated from $E$ by $\ell$.  Using this notation, let us define $\calB(w_0)=w_1$ and similarly define $\calB$ on all of $\partial\bbD$.  We want to show that if the foci of an ellipse are given in $\bbD$, then there is a unique eccentricity of the ellipse $E$ with those foci such that the corresponding map $\calB$ satisfies $\calB^n(w_0)=w_0$.  This last condition can be rephrased as
\begin{equation}\label{argsum}
\sum_{j=0}^{n-1}\left(\arg(\calB^{j+1}(w_0))-\arg(\calB^j(w_0))\right)=2\pi
\end{equation}
where each difference of arguments in this sum is taken in $(0,2\pi)$.  Our above argument for triangles proves monotonicity in the left-hand side of \eqref{argsum} as a function of the eccentricity.  We also saw above that there is a unique eccentricity $e$ such that if the eccentricity of $E$ is $e$, then $\calB^3(w_0)=w_0$ and Poncelet's Theorem implies that this value of $e$ is independent of $w_0\in\partial\bbD$.  Monotonicity implies that if the eccentricity of $E$ is smaller than $e$, then the left-hand side of \eqref{argsum} is strictly larger than $2\pi$.  Similarly, as the ellipse approaches touching the unit circle, it takes more than $n$ iterations of $\calB$ to get past the point on $E$ closest to $\partial\bbD$, so the left-hand side of \eqref{argsum} converges to something less than $2\pi$.  By continuity there is a unique eccentricity where the left-hand side of \eqref{argsum} is $2\pi$. A priori, that eccentricity could be $w_0$--dependent

One can show there is an $A\in S_n$ whose numerical range is tangent to this $n$gon at the same tangent points.  Given Kippenhahn's result, the $n$fold agreement of the ellipse and $\partial N(A)$ and a use of Bezout's theorem proves that the two curves are the same,  proving Poncelet's theorem in this case.

It would be interesting to have an OPUC understanding of Kippenhahn's Theorem.

%%%%%%%%%%%%%%%%%%%%%%%%%%%%%%%%%%%%%%%%%%%%%%%%%%%%%%%%%%%%%%
\section{From the Numerical Range to the Eigenvalues} \lb{s13}
%%%%%%%%%%%%%%%%%%%%%%%%%%%%%%%%%%%%%%%%%%%%%%%%%%%%%%%%%%%%%%

In this final section, we want to discuss how for $A\in S_n$, $N(A)$ determines the eigenvalues of $A$ (and so, up to unitary equivalence, $A$).  The results in Sections \ref{s1} and \ref{s3} provide two ways that we want to discuss first.  Then we'll turn to a potential method via algebraic geometry.

Here are the two methods:

(1) Given $n$ and $\partial N(A)$, draw two $n$gons with vertices in $\partial\bbD$ that circumscribe $N(A)$ and apply Theorem \ref{TC}/Theorem \ref{T8}.

(2) Given $n$ and $\partial N(A)$, draw one such polygon.  Use Theorem \ref{T4} to compute the $m_j$ and then Theorem \ref{TB}/Theorem \ref{T6} to determine the zeros of $\Phi_n$ which are the eigenvalues.

There is another connection between eigenvalues and a curve related to $N(A)$ that is related to the work of Kippenhahn \cite{Kipp} mentioned earlier and extended by Singer \cite{Singer} and Langer--Singer \cite{Langer}.  There is a real projective curve, $\Gamma$, so that $N(A)$ is the convex hull of $\Gamma$ and the eigenvalues are exactly the (real) foci of $\Gamma$ (foci are an involved construction from algebraic geometry that, for an ellipse, are the usual foci).

When we first learned of this result, we assumed that it must imply for the case of elliptical $N(A)$, that the eigenvalues must be its two usual foci, obviously not both simple if $n \ge 3$.  This was wrong!  Indeed, there has been a detailed  study of ellipses in the case $n=4$ by Fujimura \cite{Fuji} and Gorkin--Wagner \cite{Gorkin2017B} that shows in this case there is a third usually distinct eigenvalue.

The key point is that while $\partial N(A)$ is included in $\Gamma$, $\Gamma$ can have additional pieces inside $N(A)$ (and if $n\ge 3$, there are always such additional pieces).  As we've seen, $\partial N(A)$ is the envelope of the lines obtained by joining successive eigenvalues of the various rank one unitary dialations, $U_\lambda$.  As proven by Singer \cite[Prop. 3.1]{Singer}, $\Gamma$ is the envelope of the complete graph on the eigenvalues of $U_\lambda$, i.e. all lines joining pairs of distinct eigenvalues.  We illustrated this with a simple illuminating example.  Once one picks the eigenvalues $a, b, c \in\bbD$ for $A\in S_3$, one can form the various $\Phi_4(z;\lambda)$ and the polygons associated to their zeros.  We took examples with $b=0.8e^{34i}$ and $c=0.57e^{4i}$ (chosen by experimenting to get clean output).  For one example, we picked $a=0.7i$ and for the second we chose $a=-0.74949\dots+i 0.164697\dots$.  The later was chosen using \cite[Prop. 3.7]{Gorkin2017B} to give an elliptical $N(A)$.

Figure \ref{fig:outerpolygs} shows the outer polygons when $a=0.7i$ (left) or for the elliptical case (right).

%%%%%%%%%%%%%%%%%%%%%%%%%%%%%%%%%%%%%%%%%%%%%%%%%%%%%%%%%%
\begin{figure}[htb]
	\centering
	\begin{tabular}{ll} \hspace{0cm}\mbox{
			\includegraphics[width=0.45\textwidth]{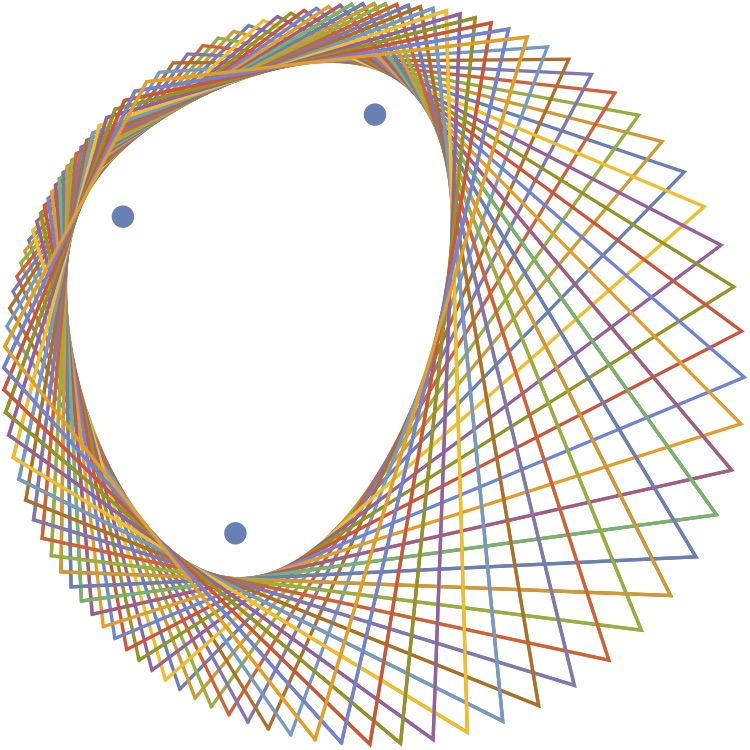}} &
		\hspace{2mm}
		\mbox{\includegraphics[width=0.45\textwidth]{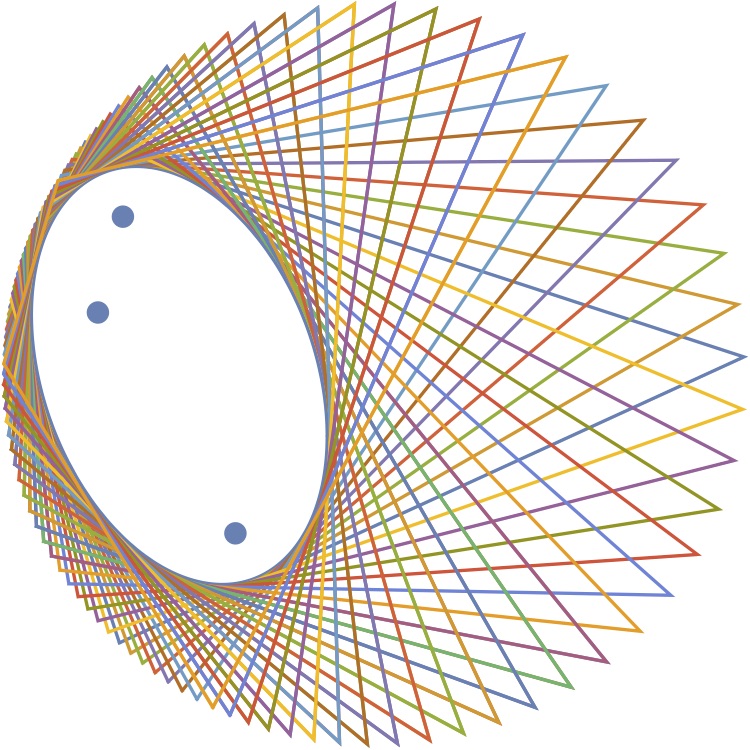}}
	\end{tabular}
	\caption{The outer polygons and the numerical range $N(A)$  for $A\in S_3$ determined by its two eigenvalues at $ 0.8e^{34i}$ and  $ 0.57e^{4i}$, and the third one either at $ 0.7i$ (left) or at $ -0.74949\dots+i 0.164697\dots$ (right). The eigenvalues are indicated by the fat dots in the interior of $N(A)$.}
	\label{fig:outerpolygs}
\end{figure}

In conformance with the discussion in section \ref{s7}, one can see the convex $\partial N(A)$ formed in each case.  Figure \ref{fig:complete} shows the complete graphs of the eigenvalues of the $U_\lambda$ when $a=0.7i$ (left)  and in the elliptical case (right).

%%%%%%%%%%%%%%%%%%%%%%%%%%%%%%%%%%%%%%%%%%%%%%%%%%%%%%%%%%
\begin{figure}[htb]
	\centering
	\begin{tabular}{ll} \hspace{0cm}\mbox{
			\includegraphics[width=0.45\textwidth]{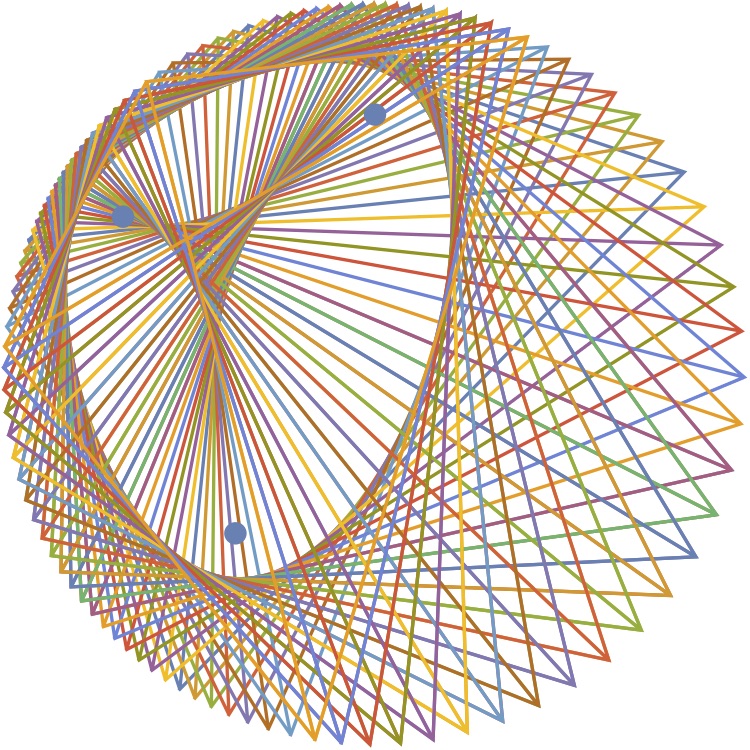}} &
		\hspace{2mm}
		\mbox{\includegraphics[width=0.45\textwidth]{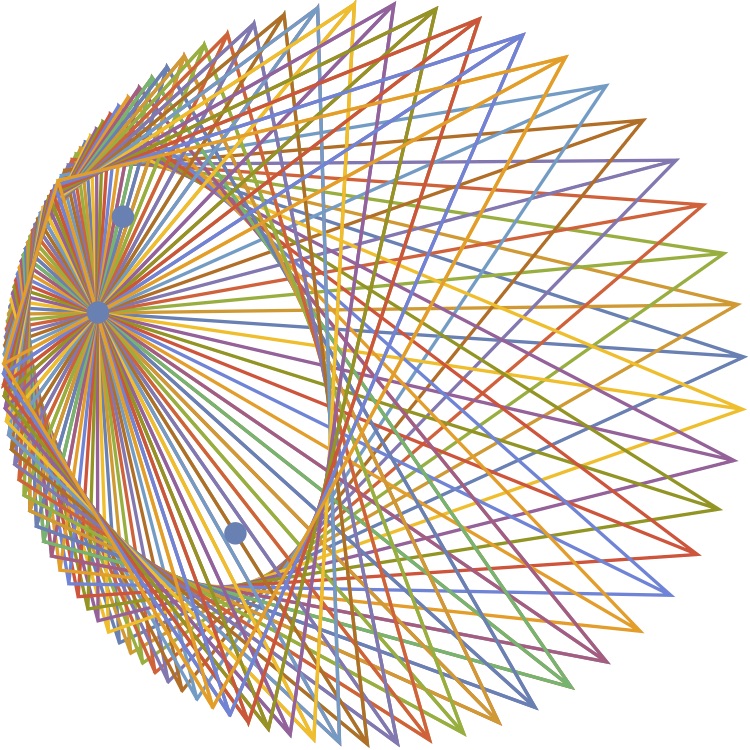}}
	\end{tabular}
	\caption{The complete graphs of the eigenvalues of the $U_\lambda$ and the numerical range $N(A)$  of two different $A\in S_3$ from Figure~\ref{fig:outerpolygs}.
	}
	\label{fig:complete}
\end{figure}

One sees that there is an extra piece of $\Gamma$ in the resulting envelope.

As shown in \cite{Gorkin2017B}, in the elliptical case, the extra piece of $\Gamma$ is the single point $\{a\}$.  As can be seen by looking at the case where all eigenvalues of $A$ are zero and the extra pieces are circles when $n \ge 4$, this single point phenomena is only true for $n=3$.  In any event, one cannot use these ideas to go directly from $N(A)$ as a set to the eigenvalues.

%%%%%%%%%%%%%%%%%%%%%%%%%%%%%%%

\end{document}